\documentclass[11pt]{article}
\usepackage{graphicx} 
\usepackage{bbm}
\usepackage{geometry}
\usepackage{color}
\usepackage{amsfonts}
\usepackage{algorithm}
\usepackage{algorithmic}
\usepackage{latexsym, amssymb, amsmath, amscd, amsthm, amsxtra}
\usepackage{mathtools}
\usepackage{enumerate}
\usepackage[all]{xy}
\usepackage{mathrsfs}
\usepackage{fancyhdr}
\usepackage{listings}
\usepackage{hyperref}
\usepackage{enumitem}

\def\11{\mathbbm{1}}

\def\ER{Erd\H{o}s-R\'enyi\ }

\def\Fop{\operatorname{F}}

\newcommand{\op}{\operatorname{op}}

\newcommand{\Pb}{\mathbb P}

\newtheorem{thm}{Theorem}[section]

\newtheorem{proposition}[thm]{Proposition}

\newtheorem{lemma}[thm]{Lemma}
\newtheorem{cor}[thm]{Corollary}

\newtheorem{claim}[thm]{Claim}

\newtheorem{remark}[thm]{Remark}

\numberwithin{equation}{section}

\makeatletter
\newenvironment{breakablealgorithm}
{
		\begin{center}
			\refstepcounter{algorithm}
			\hrule height.8pt depth0pt \kern2pt
			\renewcommand{\caption}[2][\relax]{
				{\raggedright\textbf{\ALG@name~\thealgorithm} ##2\par}%
				\ifx\relax##1\relax 
				\addcontentsline{loa}{algorithm}{\protect\numberline{\thealgorithm}##2}%
				\else 
				\addcontentsline{loa}{algorithm}{\protect\numberline{\thealgorithm}##1}%
				\fi
				\kern2pt\hrule\kern2pt
			}
		}{
		\kern2pt\hrule\relax
	\end{center}
}
\makeatother

\hypersetup{colorlinks=true,linkcolor=blue}
\title{The Umeyama algorithm for matching correlated Gaussian geometric models in the low-dimensional regime}
\author{Shuyang Gong\\Peking University\and Zhangsong Li\\Peking University}
\date{\today}
\begin{document}
\maketitle

\begin{abstract}
    Motivated by the problem of matching two correlated random geometric graphs, we study the problem of matching two Gaussian geometric models correlated through a latent node permutation. Specifically, given an unknown permutation $\pi^*$ on $\{1,\ldots,n\}$ and given $n$ i.i.d.\ pairs of correlated Gaussian vectors $\{X_{\pi^*(i)},Y_i\}$ in $\mathbb{R}^d$ with noise parameter $\sigma$, we consider two types of (correlated) weighted complete graphs with edge weights given by $A_{i,j}=\langle X_i,X_j \rangle$, $B_{i,j}=\langle Y_i,Y_j \rangle$. 
    The goal is to recover the hidden vertex correspondence $\pi^*$ based on the observed matrices $A$ and $B$. 
    For the low-dimensional regime where $d=O(\log n)$, 
    Wang, Wu, Xu, and Yolou \cite{WWXY22+} established the information thresholds for exact and almost exact recovery in matching correlated Gaussian geometric models. They also conducted numerical experiments for the classical Umeyama algorithm.
    In our work, we prove that this algorithm achieves exact recovery of $\pi^*$ when the noise parameter $\sigma=o(d^{-3}n^{-2/d})$, and almost exact recovery when $\sigma=o(d^{-3}n^{-1/d})$. 
    Our results approach the information thresholds up to a $\operatorname{poly}(d)$ factor in the low-dimensional regime.  
\end{abstract}

\section{Introduction}

In this paper, we consider the problem of matching two types of correlated Gaussian geometric models, namely, the {\em dot-product model} and the {\em distance model}. Specifically, we consider the Gaussian model
\begin{equation}{\label{eq-def-correlated-vector}}
    Y_i = X_{\pi^*(i)} + \sigma Z_i \mbox{ for } i=1,\ldots,n \,,
\end{equation}
where $X_i,Z_i$'s are i.i.d.\ standard $d$-dimensional Gaussian vectors and $\pi^*$ is an unknown permutation on $[n]=\{1,\ldots,n\}$. 
In the matrix form we have
\begin{equation}{\label{eq-def-correlated-Gaussian-matrix}}
    Y=\Pi^* X+\sigma Z \,,
\end{equation}
where $X,Y,Z \in \mathbb{R}^{n*d}$ are matrices whose rows are $X_i$'s, $Y_i$'s and $Z_i$'s respectively, $\Pi^* \in \mathfrak{S}_n$ denotes the permutation matrix corresponding to $\pi^*$ (here $\mathfrak{S}_n$ is the collection of all permutation matrices). For the dot-product model, the observations are two Wishart matrices $A=XX^{\top}$ and $B=YY^{\top}$. For the distance model, the observations are two matrices $\widehat A$ and $\widehat B$ with edge weights given by $\widehat A_{i,j}=\|X_i-X_j\|_2$ and $\widehat B_{i,j}=\|Y_i-Y_j\|_2$.



Our goal is to recover the latent permutation matrix $\Pi^*$ given the observations $A$ and $B$ (respectively, $\widehat{A}$ and $\widehat{B}$). In the following we consider the classical Umeyama algorithm proposed in \cite{Umeyama88} for solving graph matching type of problems efficiently (see also \cite[Section 3]{WWXY22+} for a thorough discussion on the intuition behind this algorithm as well as valuable numerical experiments). This algorithm takes two $n\!*\!n$ symmetric matrices $A,B$ with rank at most $d$ as inputs and performs the following procedure: let
\begin{equation}{\label{eq-spectral-decomposition-A-B}}
    A=\sum_{i=1}^{d} \lambda_i u_i u_i^{\top} \mbox{ and } B=\sum_{i=1}^{d} \mu_i v_i v_i^{\top}
\end{equation}
be the spectral decompositions of $A$ and $B$ respectively, where $\lambda_1 \geq \ldots \geq \lambda_d$ and $\mu_1 \geq \ldots \geq \mu_d$. Consider the following estimator
\begin{equation}{\label{eq-def-Umeyama-estimator}}
    \widehat{\Pi}_{\operatorname{Umeyama}} = \arg \max_{\Pi \in \mathfrak{S}_n} \Bigg\{ \max_{s \in \{-1,+1\}^d} \Bigg\langle \Pi, \sum_{i=1}^{d} s_i v_i u_i^{\top} \Bigg\rangle \Bigg\} \,.
\end{equation}
Denoting $U=[u_1,\ldots,u_d]$ and $V=[v_1,\ldots,v_d]$ respectively,
this estimator can also be expressed in the following matrix form
\begin{equation}{\label{eq-Umeyama-estimator-matrix-form}}
    \widehat{\Pi}_{\operatorname{Umeyama}} = \arg \max_{\Pi \in \mathfrak{S}_n} \Bigg\{ \max_{\Psi \in \mathcal{S}_d} \Big\langle \Pi U \Psi, V \Big\rangle \Bigg\} \,,
\end{equation}
where $\mathcal{S}_d$ is the family of $d\!*\!d$ diagonal matrices with diagonal entries in $\{-1,+1\}$. We formally present this algorithm as follows:
\begin{breakablealgorithm}{\label{Umeyama Algorithm}}
	\caption{Umeyama Algorithm}
    \begin{algorithmic}[1]
    \STATE \textbf{Input}: Symmetric matrices $A$ and $B$.
    \STATE Calculate the spectral decompositions \eqref{eq-spectral-decomposition-A-B}.
    Write $U=[u_1,\ldots,u_d]$ and $V=[v_1,\ldots,v_d]$ respectively.
    \FOR{$\Psi \in \mathcal{S}_d$}
    \STATE Find $\widehat{\Pi}(\Psi)= \arg\max \{ \langle \Pi U \Psi, V \rangle : \Pi\in\mathfrak S_n \}$.
    \STATE Record $\operatorname{MAX}(\Psi)=\langle \widehat{\Pi}(\Psi) U \Psi,V \rangle$.
    \ENDFOR
    \STATE Find $\Psi^*=\arg\max \{ \operatorname{MAX}(\Psi): \Psi \in \mathcal{S}_d \}$.
    \STATE \textbf{Output}: $\widehat{\Pi}_{\operatorname{Umeyama}}= \widehat{\Pi}(\Psi^*)$. 
    \end{algorithmic}
\end{breakablealgorithm}
This matching algorithm indeed runs in polynomial time due to the fact that we are under the low-dimensional regime $d=O(\log n)$ which implies $\operatorname{card}(\mathcal{S}_d)=2^{d}=n^{O(1)}$ and the fact that for each fixed $\Psi$, the procedure of maximizing over $\Pi\in\mathfrak S_n$ can be reduced to a {\em linear assignment problem}, which can be solved in time $O(n^3)$ by a linear program (LP) over doubly stochastic matrices or by the Hungarian algorithm \cite{Kuhn55}. Thus the total running time of this algorithm is $O(2^d n^{3})$. Our main results are stated as follows:

\begin{thm}[Almost exact recovery]{\label{thm-almost-exact-recovery}}
Assume $5\leq d = O(\log n)$ and $\sigma = o(d^{-3} n^{-1/d})$ and consider the correlated dot-product model $(A, B)$.
By taking $(A,B)$ as the input in the Umeyama algorithm and obtaining $\widehat{\Pi}_{\operatorname{Umeyama}}$ as the output, we have with probability $1-o(1)$ that
\begin{equation}\label{eq-bound-epsilon-for-partial-recovery}
    d_{\operatorname{H}}( \widehat{\Pi}_{\operatorname{Umeyama}}, \Pi^*) \leq \frac{50n}{ d \log [( \sigma^{-1}d^{-3} n^{-1/d} \wedge \log n ) ] }\,,
\end{equation}
  where $d_{\operatorname{H}}(\widehat{\Pi},\Pi^*)=\sum_{i, j} \mathbf 1_{\widehat \Pi_{ij} \neq \Pi^*_{ij}}$ is the Hamming distance between $\widehat{\Pi}$ and $\Pi^*$. 
\end{thm}

\begin{thm}[Exact recovery]\label{thm-exact-recovery}
Assume $5\leq d = O(\log n)$ and $\sigma = o(d^{-3} n^{-2/d})$ and consider the correlated dot-product model $(A, B)$.
By taking $(A,B)$ as the input in the Umeyama algorithm and obtaining $\widehat{\Pi}_{\operatorname{Umeyama}}$ as the output, we have $\widehat{\Pi}_{ \operatorname{Umeyama}} = \Pi^*$ with probability $1-o(1)$. 
\end{thm}
\begin{remark}
    Actually the proofs of Theorems~\ref{thm-almost-exact-recovery} and \ref{thm-exact-recovery} remain unchanged even when $d\leq n^{\delta}$ for some small constant $\delta$, though in this case the Umeyama algorithm does not run in polynomial time.
    In addition, we point out that the assumption $d \geq 5$ in Theorems~\ref{thm-almost-exact-recovery} and \ref{thm-exact-recovery} is posed for technical reasons. We also point out that when $d$ is a constant, one may modify this algorithm to output a refined estimator $\Tilde{\Pi}$ defined as
    $$ \Tilde{\Pi} = \arg \max_{ \Pi \in \mathfrak{S}_n } \max_{ \Psi \in N(d) } \Big\langle \Pi U \Psi, V \Big\rangle \,, $$
    where $N(d)$ is a $\epsilon$-net of the set of all $d\!*\!d$ orthogonal matrices with $\epsilon = n^{-O(1)}$. This algorithm still has polynomial running time and its success is guaranteed by \cite{WWXY22+}. However, this algorithm requires running time $n^{O(d)}$, which is significantly less efficient than the Umeyama algorithm and becomes pseudo-polynomial when $d=\omega(1)$. 
\end{remark}
\begin{remark}
    It is worth noting that our results combined with \cite{WWXY22+} imply that there is no information-computation gap when $d=O(1)$. It remains an intriguing question to see if this phenomenon occurs when $d=\omega(1)$.
\end{remark}

We now consider the distance model. It turns out that it will be more convenient to consider $(\widehat{A} \circ \widehat{A})_{i,j}=\|X_i-X_j\|^2_2$ and $(\widehat{B} \circ \widehat{B})_{i,j}=\|Y_i-Y_j\|_2^2$, where $\circ$ is the Hadamard product. 
For ease of analysis, we let $\mathbf F=\frac{1}{n}\mathbf{1}\mathbf{1}^\top$ and consider 
\begin{equation}\label{eq-goal-distance-model}
    \widetilde A=-\frac{1}{2}(\operatorname{I}_n-\mathbf F)(\widehat{A} \circ \widehat{A})(\operatorname{I}_n-\mathbf F) \mbox{ and }
    \widetilde B=-\frac{1}{2}(\operatorname{I}_n-\mathbf F)(\widehat{B} \circ \widehat{B})(\operatorname{I}_n-\mathbf F) \,.
\end{equation}
We apply Algorithm~\ref{Umeyama Algorithm} to $\widetilde A$ and $\widetilde B$. Our results in this case are as follows: 

\begin{thm}[Almost exact recovery of distance model]{\label{thm-almost-exact-recovery-distance}}
Assume $5\leq d = O(\log n)$ and $\sigma = o(d^{-3} n^{-1/d})$ and consider the correlated distance model $(A, B)$.
By taking $(\widetilde{A},\widetilde{B})$ as the input in the Umeyama algorithm and obtaining $\widehat{\Pi}_{\operatorname{Umeyama}}$ as the output, we have with probability $1-o(1)$ that
\begin{equation}{\label{eq-bound-epsilon-for-partial-recovery-distance}}
    d_{\operatorname{H}}( \widehat{\Pi}_{\operatorname{Umeyama}}, \Pi^*) \leq \frac{50n}{ d \log (d^{-3}n^{-1/d} \sigma^{-1} \wedge\log n) } \,.
\end{equation}

\end{thm}

\begin{thm}[Exact recovery of distance model]\label{thm-exact-recovery-distance}
Assume $5\leq d = O(\log n)$ and $\sigma = o(d^{-3} n^{-2/d})$ and consider the correlated distance model $(A, B)$. By taking $(\widetilde{A},\widetilde{B})$ as the input in the Umeyama algorithm and obtaining $\widehat{\Pi}_{\operatorname{Umeyama}}$ as the output, we have $\widehat{\Pi}_{ \operatorname{Umeyama}} = \Pi^*$ with probability $1-o(1)$. 
\end{thm}

\subsection{Backgrounds and related works}

Graph matching (or graph alignment) refers to finding the vertex correspondence between two graphs such that the total number of common edges is maximized. It plays essential roles in various applied fields such as computational biology \cite{SXB08, VCL+15}, social networking \cite{NS08, NS09}, computer vision \cite{BBM05, CSS06} and natural language processing \cite{HNM05}. This problem is an instance of {\em Quadratic assignment problem} which is known to be NP-hard in the worst case, making itself an important and challenging combinatorial optimization problem. From the theoretical point of view, perhaps the most widely studied model is the {\em correlated \ER graph model} \cite{PG11}, where the observations are two \ER graphs with correlated pairs of edges through an underlying vertex bijection $\pi^*$. In recent years, many results have been obtained concerning this model with emphasis placed on the two important and entangling issues\textemdash the information threshold (i.e., the statistical threshold) and the computational phase transition. On one hand, the collaborative endeavors of the community, as evidenced in \cite{CK16, CK17, HM23, GML21, WXY22, WXY23, DD23a, DD23b}, have led to an essentially comprehensive understanding of information thresholds in correlation detection and vertex matching. On the other hand, continual advancements in algorithms for both problems are evident in works such as \cite{PG11, YG13, LFP14, KHG15, FQRM+16, SGE17, BCL+19, DMWX21, FMWX23a, FMWX23b, BSH19, CKMP19, DCKG19, MX20, GM20, GML20+, MRT21, MRT23, MWXY21+, GMS22+, MWXY23, DL22+, DL23+}. The current state of the art in algorithms can be summarized as follows: in the sparse regime, efficient matching algorithms are available when the correlation exceeds the square root of Otter’s constant (which is approximately 0.338) \cite{GML20+, GMS22+, MWXY21+, MWXY23}; in the dense regime, efficient matching algorithms exist as long as the correlation exceeds an arbitrarily small constant \cite{DL22+, DL23+}.

Roughly speaking, the separation between the sparse and dense regimes mentioned above depends on whether the average degree grows polynomially or sub-polynomially.
In addition, although it seems rather elusive to prove hardness on the graph matching problem for a typical instance under the assumption of P$\neq$NP, in \cite{DDL23+} the authors presented evidences that the state-of-the-art algorithms indeed capture the correct computation thresholds based on the analysis of a specific class known as low-degree polynomial algorithms. Along this line, efforts have been made also on the (closely related) random optimization problem of maximizing the overlap between two independent \ER graphs (see \cite{DDG22, DGH23}).

Since our current understanding for the information-computation transition of the correlated \ER model is more or less satisfactory, an important future direction is to  understand the matching problem in other important correlated random graph models such as the random growing graphs \cite{RS22}, the stochastic block model \cite{RS21} and the inhomogeneous random graphs \cite{DFW23+}. 
Still, these models with independent edge occurrences fail to capture the complicated dependence structure in real-world networks. As an important example for modeling realistic networks, the random graph geometric model incorporates edge correlations via some underlying spatial geometry \cite{AG14}: given random vectors $\{X_i\}$, the edge connection of the graph is determined by
thresholding $\langle X_i,X_j \rangle$ and thus is strongly related to the model considered in this paper. 
While the Gaussian geometric models are not graph models in the most classical sense, they do share many important features of random geometric graphs including modeling edge dependence. We believe that it is a necessary and substantial step to understand the Gaussian geometric models before one analyzes the more challenging geometric graph model.
In \cite{WWXY22+}, inspired by the analysis of the stronger linear assignment model \cite{DCK19, DCK20, KNW22}, the authors obtained the informational thresholds for Gaussian geometric models up to some constant factor: basically they showed that in the low-dimensional regime where $d=o(\log n)$, it is possible to obtain an exact recovery (respectively, almost exact recovery) of the latent permutation when the noise parameter $\sigma=o(n^{-2/d})$ (respectively, when $\sigma=o(n^{-1/d})$).
In addition, they observed by numerical experiments that the Umeyama algorithm indeed behaves well in the low-dimensional regime. Motivated by their results we provide rigorous analysis of this algorithm, showing that it achieves exact recovery (respectively, almost exact recovery) when the noise parameter $\sigma$ is below the information thresholds up to a $\operatorname{poly}(d)$ factor.

\subsection{Notations}
Let $W_1(\mu,\nu)$ be the Wasserstein distance between two probability distributions $\mu$ and $\nu$, namely
\begin{equation}{\label{eq-def-Wasserstein-distance}}
    W_1(\mu,\nu) := \inf_{\gamma \in \Gamma(\mu,\nu)} \mathbb{E}_{(x,y)\sim\gamma} \Big[ \|x-y\|_2 \Big] \,, 
\end{equation}
where $\Gamma(\mu,\nu)$ is the set of all couplings between $\mu$ and $\nu$. In addition, for two random vectors $X,Y$, let $W_1(X,Y)=W_1(\mu_X,\mu_Y)$, where $\mu_X,\mu_Y$ are the probability measures induced by $X,Y$ respectively. We denote by $\mathcal{N}(\mu,\sigma^2)$ the normal random variable with mean $\mu$ and variance $\sigma^2$.
   
   For two $l\!*\!m$ matrices $M_1$ and $M_2$, we define their inner product to be
\[
\langle M_1,M_2\rangle:=\sum_{i=1}^l \sum_{j=1}^m M_1(i,j)M_2(i,j) \,.
\]
In addition, for an $l\!*\!m$ symmetric matrix $M$, denote $\lambda_1(M) \geq \ldots \geq \lambda_{l \wedge m}(M)$ as its singular values, and denote
\begin{align*}
    \| M \|_{\operatorname{F}} = \langle M,M \rangle^{\frac{1}{2}} = \Big( \sum_{i=1}^{l \wedge m} \lambda_i(M)^2 \Big)^{\frac{1}{2}}, \
    \| M \|_{\operatorname{op}} = \lambda_1(M), \
    \| M \|_{\infty} = \max_{ \substack{ 1 \leq i \leq l \\ 1 \leq j \leq m } } |M_{i,j}| 
\end{align*}
to be its Frobenius norm, operator norm and $\infty$-norm respectively. Furthermore, if $M$ is symmetric, we denote $\tau_1(M) \geq \ldots \geq \tau_{l}(M)$ to be the eigenvalues of $M$. We also denote $\operatorname{I}_{m}$ as the $m\!*\!m$ identity matrix.

For any two positive sequences $\{a_n\}$ and $\{b_n\}$, we write equivalently $a_n=O(b_n)$, $b_n=\Omega(a_n)$, $a_n\lesssim b_n$ and $b_n\gtrsim a_n$ if there exists a positive absolute constant $c$ such that $a_n/b_n\leq c$ holds for all $n$. We write $a_n=o(b_n)$, $b_n=\omega(a_n)$, $a_n\ll b_n$, and $b_n\gg a_n$ if $a_n/b_n\to 0$ as $n\to\infty$.

\section{Proofs of the main theorems}

This section is devoted to the proofs of Theorems~\ref{thm-almost-exact-recovery}, \ref{thm-exact-recovery}, \ref{thm-almost-exact-recovery-distance} and \ref{thm-exact-recovery-distance}. For notational convenience, in the proofs of Theorems~\ref{thm-almost-exact-recovery} and \ref{thm-almost-exact-recovery-distance} (respectively, Theorems~\ref{thm-exact-recovery} and \ref{thm-exact-recovery-distance}) we assume $\sigma=d^{-3} n^{-\frac{1}{d}}/\omega_n$ (respectively, $\sigma=d^{-3} n^{-\frac{2}{d}}/\omega_n$) for some fixed sequence $\omega_n$ with $\omega_n\to \infty$. We will also define
\begin{equation}{\label{eq-def-widehat-omega}}
    \widehat{\omega}_n = \sqrt{\omega_n \wedge \log n}\quad\text{and} \quad \widehat{\sigma} = \sigma \omega_n/\widehat{\omega}_n \,.
\end{equation}
Here $\widehat{\omega}_n$ and $\widehat{\sigma}$ serve as truncations to avoid certain technical issues when $\omega_n$ grows too quickly towards infinity. Clearly we still have $1\ll\widehat{\omega}_n$ and $\widehat{\sigma}$ remains below the desired threshold (i.e. $d^{-3} n^{-1/d}$ or $d^{-3} n^{-2/d}$) with a slack factor of $\widehat \omega_n$. 
In the following subsections, our focus is on proving Theorems~\ref{thm-almost-exact-recovery} and \ref{thm-exact-recovery}. Given the similarity between the proofs of the dot-product model and the distance model, for Theorems~\ref{thm-almost-exact-recovery-distance} and \ref{thm-exact-recovery-distance}, we will provide an outline with the main differences while adapting arguments from Theorems~\ref{thm-almost-exact-recovery} and \ref{thm-exact-recovery} without presenting full details.

\subsection{Comparison between right singular vectors}{\label{subsection:comparison-SVD}}
We first consider the dot-product model. Denote
\begin{equation}{\label{eq-def-SVD}}
    \mathsf{SVD}(X)=U {\Lambda} Q^{\top} \mbox{ and } \mathsf{SVD}(Y)=V {\Sigma} R^{\top}
\end{equation}
as the singular value decompositions of $X$ and $Y$ respectively, where $Q,R$ are $d*d$ orthogonal matrices and $\Lambda,\Sigma$ are $d\!*\!d$ diagonal matrices such that $\Lambda_{1,1} \geq \ldots \geq \Lambda_{d,d} \geq 0$ and $\Sigma_{1,1} \geq \ldots \geq \Sigma_{d,d} \geq 0$. As will be shown, the main technical input is the following proposition, which reveals that $Q$ and $R$ are ``relatively close'' up to the multiplication of a matrix in $\mathcal{S}_d$. 
Define $\mathcal E_0$ to be the event that there exists $\Psi_0 \in \mathcal{S}_d$ such that $\| Q \Psi_0 - R \|_{\Fop} \leq d^3 \widehat{\sigma}$.
\begin{proposition}\label{prop-diff-orthogonal-matrix}
     We have $\Pb(\mathcal E_0)=1-o(1)$. 
\end{proposition}
\begin{proof}
We write $Q=[q_1,\ldots,q_d]$ and $R=[r_1,\ldots,r_d]$. From the SVD decompositions of $X$ and $Y$ defined in \eqref{eq-def-SVD}, the spectral decompositions of $X^{\top}X$ and $Y^{\top}Y$ can be written as 
\begin{equation}
    X^{\top}X = Q \Lambda^2 Q^{\top} = \sum_{i=1}^{d} \lambda_i q_i q_i^{\top}, \quad Y^{\top}Y = R \Sigma R^{\top}  = \sum_{i=1}^{d} \mu_i r_i r_i^{\top} \,,
\end{equation}
where $\lambda_i = \Lambda_{i,i}^2$ and $\mu_i = \Sigma_{i,i}^2$.
 Define $\angle(q_i,r_i)$ to be the angle between vectors $q_i$ and $r_i$. The following crucial claim helps us control the dot product between $q_i$ and $r_i$.
\begin{claim}{\label{lem-bound-angle}}
    With probability $1-o(1)$ we have $|\operatorname{sin} \angle(q_i,r_i)| \leq d^2 \sigma \sqrt{\widehat{\omega}_n} $ for $1 \leq i \leq d$.
\end{claim}
\begin{proof}[Proof of Proposition~\ref{prop-diff-orthogonal-matrix} assuming Claim~\ref{lem-bound-angle}]
From Claim~\ref{lem-bound-angle} we obtain $1-\langle q_i,r_i \rangle^2 = \operatorname{sin}^2 \angle(q_i,r_i) \leq d^4 \sigma^2 \widehat{\omega}_n$. Thus, by taking $\psi_i=\operatorname{sgn} \langle q_i,r_i \rangle$, we have 
\begin{align*}
    \big| \psi_i-\langle q_i,r_i \rangle \big| = \big| 1-\psi_i \langle q_i,r_i \rangle \big| = 1 - \big|\langle q_i,r_i \rangle \big|=\frac{1-|\langle q_i,r_i\rangle|^2}{1+|\langle q_i,r_i\rangle|} \leq d^4 \sigma^2 {\widehat{\omega}_n} \,.
\end{align*}
Also, using the orthogonality of $\{q_i : 1 \leq i \leq d\}$ and $\{r_i : 1 \leq i \leq d\}$ we have for $i \neq j$
\begin{align*}
    \langle q_i,r_j \rangle^2 = \langle q_i-\psi_i r_i,r_j \rangle^2 \leq \|q_i-\psi_i r_i\|^2 = 2-2\psi_i \langle q_i,r_i \rangle \leq 2d^4 \sigma^2 {\widehat{\omega}_n} \,,
\end{align*}
where the first inequality follows from Cauchy-Schwartz inequality.
In conclusion, we have shown that for $\Psi_0 = \operatorname{diag}(\psi_i)_{i=1}^{d}$, we have 
\begin{align*}
    & \|Q\Psi_{0}-R\|_{\operatorname{F}}^2 = \|\Psi_{0}-Q^{\top}R\|_{\operatorname{F}}^2 = \sum_{i=1}^{d} (\psi_i-\langle q_i,r_i \rangle)^2 + \sum_{i \neq j} \langle q_i,r_j \rangle^2 \\
    \leq \ & d \cdot (2d^4 \sigma^2 \widehat{\omega}_n)^2 + d(d-1) \cdot 2d^4 \sigma^2 \widehat{\omega}_n \leq 2d^6 \sigma^2 \widehat{\omega}_n \overset{\eqref{eq-def-widehat-omega}}{=}d^6\widehat{\sigma}^2\frac{2\widehat\omega_n^3}{\omega_n^2}\leq (d^3 \widehat{\sigma})^2 \,.\qedhere
\end{align*}
\end{proof}
The rest of this subsection is devoted to the proof of Claim~\ref{lem-bound-angle}.
We view $q_i$ and $r_i$ as the unit eigenvectors of the centered matrix $\mathbf X=X^{\top}X-n \operatorname{I}_d$ and $\mathbf Y=Y^{\top}Y-n(1+\sigma^2) \operatorname{I}_d$ respectively and apply the Davis-Kahan theorem in \cite{DK70} (see also \cite[Theorem 1]{YWS14}). This yields that  
\begin{align*}
    \big| \operatorname{sin} \angle(q_i,r_i) \big| \leq \frac{ 2 \| \mathbf X - \mathbf Y \|_{\operatorname{op}} }{ \delta(\mathbf X) } \,,
\end{align*}
where $\delta(\mathbf X)$ is the smallest eigen-gap of $\mathbf X$ and is defined as $\min \{ |\tau_i(\mathbf X) - \tau_j(\mathbf X)| : i \neq j \}$. 
By the triangle inequality, we have that with probability $1-o(1)$,
\begin{align*}
    \|\mathbf Y - \mathbf X\|_{\operatorname{op}} &= \| (X+\sigma Z)^{\top} (X+\sigma Z) - X^{\top}X - \sigma^2 n \operatorname{I}_d \|_{\operatorname{op}} \\
    &\leq \sigma^2 \| Z^{\top}Z - n \operatorname{I}_d \|_{\operatorname{op}} + 2\sigma \| X^{\top} Z \|_{\operatorname{op}} \\
    &\leq \sigma^2 \| Z^{\top}Z - n\operatorname{I}_d \|_{\operatorname{F}} + 2\sigma \| X^{\top} Z \|_{\operatorname{F}} \leq 10d \sigma \sqrt{n} \widehat{\omega}_n^{0.1} \,,
\end{align*}
where the last inequality follows from $\mathbb{E} \big[ \| Z^{\top}Z - n\operatorname{I}_d \|_{\operatorname{F}}^2 \big], \mathbb{E}\big[ \| X^{\top} Z \|_{\operatorname{F}}^2 \big] \leq 2d^2 n$ and Chebyshev's inequality.
Now it remains to show with probability $1-o(1)$
\begin{equation}{\label{eq-goal-spectral-gap}}
    \min \big\{ |\tau_i(\mathbf X) - \tau_j(\mathbf X)| : i \neq j \big\} \geq \frac{1}{2d} \sqrt{n}/\widehat{\omega}_n^{0.1} \,.
\end{equation}
To this end, we let $\mathsf G$ be a standard $d\!*\!d$ Gaussian orthogonal ensemble (i.e., the real symmetric matrix with i.i.d. $\mathcal{N}(0,1)$ off-diagonal entries and i.i.d. $\mathcal{N}(0,2)$ diagonal entries) and prove the following two statements:
\begin{enumerate}
    \item[(1)] $\delta(\mathsf{G})$ is lower-bounded by $1/(d\widehat{\omega}_n^{0.1})$ with probability $1-o(1)$.
    \item[(2)] The $1$-Wasserstein distance between $\mathbf{X}$ and $\sqrt{n} \mathsf{G}$ is bounded by $3d$.
\end{enumerate}
We begin by completing the proof provided with (1) and (2). By (2) there exists a coupling $(\mathbf X, \sqrt{n} \mathsf{G})$ such that $\mathbb E \| \mathbf X - \sqrt{n} \mathsf{G} \|_{\operatorname{F}} \leq 3d$. Thus by Markov's inequality with probability $1-o(1)$ we have $\| \mathbf X - \sqrt{n} \mathsf{G} \|_{\operatorname{F}} \leq d \log n$. Combining this result with the min-max characterization of eigenvalues

\begin{equation}
    \tau_i(\mathsf A)=\max_{\dim(E)=i}\min_{\substack{\|x\|_2=1,\\x\in E}}\langle \mathsf Ax,x\rangle\,,\notag
\end{equation}
 it follows that with probability $1-o(1)$ (recall $d=O(\log n)$)
\begin{align*}
    \big| \tau_i(\mathbf X) - \tau_i(\sqrt{n}\mathsf G) \big| \leq d \log n \overset{\eqref{eq-def-widehat-omega}}{\leq} \sqrt{n}/(4d\widehat{\omega}_n^{0.1}) \,.
\end{align*}
Since (1) implies
$$ \min_{i \neq j} \big\{ |\tau_i(\sqrt{n} \mathsf G) - \tau_j(\sqrt{n} \mathsf G)| \big\} \geq \sqrt{n}/(d\widehat{\omega}_n^{0.1}) \textup{ with probability } 1-o(1) \,, $$ 
we immediately get that
$$ \min_{i \neq j} \big\{ |\tau_i(\mathbf X) - \tau_j(\mathbf X)| \big\} \geq \sqrt{n}/(2d\widehat{\omega}_n^{0.1}) \textup{ with probability } 1-o(1) \,, $$
leading to \eqref{eq-goal-spectral-gap}.

Now it remains to prove (1) and (2). For (1), we may assume that $d=d_n \to \infty$ since otherwise $1/(d\widehat{\omega}_n^{0.1}) \to 0$ and for fixed $d$ all eigenvalues of $\mathsf G$ are distinct with probability $1$. Using \cite[Corollary 1]{FTW19}, for any constant $0<A<B<\infty$, we have
    \begin{align*}
        &\lim_{n\to\infty} \Pb \Big( \min_{1\leq i\neq j\leq d_n} |\tau_i(\mathsf G)-\tau_j(\mathsf G)| \geq 1 /(d_n \widehat{\omega}_n^{0.1}) \Big)\\
        \geq & \lim_{n\to\infty} \Pb\Big( \min_{1\leq i\neq j\leq d_n} d_n|\tau_i(\mathsf G)-\tau_j(\mathsf G)| \in [A,B] \Big) = \int_A^B 2xe^{-x^2} \mathrm{d}x \,.
    \end{align*}
    Since the preceding inequality holds for all $A$ and $B$, sending $A \to 0$ and $B\to\infty$ yields that
    \begin{equation*}
    \lim_{n\to\infty} \Pb\Big( \min_{1\leq i\neq j\leq d_n} |\tau_i(\mathsf G)-\tau_j(\mathsf G)| \geq 1/(d_n \widehat{\omega}_n^{0.1}) \Big) =1\,.
    \end{equation*}
    Now we prove (2). Define the matrix-valued function $f(\cdot):\mathbb R^{nd} \to \mathbb R^{d(d+1)/2}$ to be 
    \begin{equation}
        f_{i,j}(M)=
        \begin{cases}
            \frac{1}{\sqrt{n}} \sum_{k=1}^n (M_{k,i}M_{k,j}),  &i<j\,, \\
            \frac{1}{\sqrt{2n}} \sum_{k=1}^n (M_{k,i}^2-1) ,   &i=j\,.
        \end{cases}
    \end{equation}
    Let $\mathbf{X}^{\operatorname{up}}:=\sqrt{2n}f(X)$, and we see that (and we define $\mathsf G^{\operatorname{up}}$ as below)
    \begin{align*}
        \mathbf{X}^{\operatorname{up}}_{i,j}=
        \begin{cases}
            \mathbf{X}_{i,j},  &i=j\,,  \\
            \sqrt{2} \mathbf{X}_{i,j}, &i<j \,.
        \end{cases} \quad \quad \quad
        \mathsf{G}^{\operatorname{up}}_{i,j}:=
        \begin{cases}
            \mathsf{G}_{i,j}, &i=j\,,\\
            \sqrt{2}\mathsf{G}_{i,j}, &i<j \,.
        \end{cases}
    \end{align*}
    It is not hard to see that 
    \begin{align*}
        \big\| \mathbf{X}-\mathsf{G} \big\|_{\operatorname{F}} = \big\| \mathbf{X}^{\operatorname{up}} - \mathsf{G}^{\operatorname{up}} \big\|_{2} := \Big( \sum_{i \leq j} \big( \mathbf{X}^{\mathrm{up}}_{i,j} - \mathsf{G}^{\mathrm{up}}_{i,j} \big)^2 \Big)^{\frac{1}{2}} \,.
    \end{align*}
     In what follows we consider vectorization of $f(X)$, $X$ and $\mathsf G^{\operatorname{up}}$ in the following manner: we regard $X$ as an $nd$-dimensional vector with i.i.d.\ standard Gaussian entries, regard $f(X)$ as a $d(d+1)/2$-dimensional vector, and regard $\frac{1}{\sqrt{2}} \mathsf{G}^{\operatorname{up}}$ as a $d(d+1)/2$-dimensional standard Gaussian vector. Then we can apply \cite[Proposition 4.3]{NPR10} to get that (note that $\mathbb{E}[ f_{i,j}^2(X) ]=1$ and $\mathbb{E}[ f_{i,j}(X) f_{i',j'}(X) ]=0$ for $(i, j) \neq (i', j')$)
    \begin{align}
        W_1( f(X), \frac{1}{\sqrt{2}} \mathsf{G}^{\mathrm{up}} ) \leq \sqrt{ \sum_{ (i,j),(i',j') } \mathbb{E} \Big[ \Big( \mathbf{1}_{(i,j)=(i',j')} - T_{(i,j);(i',j')}(X) \Big)^2 \Big] } \,, \label{eq-bound-W-1-dist}
    \end{align}
    where the summation is taken over all $(i,j)$ and $(i',j')$ such that $1 \leq i \leq j \leq d, 1 \leq i' \leq j' \leq d$, and
    \begin{align}
    T_{(i,j);(i',j')}(x) &= \int_0^1\frac{1}{2\sqrt{t}} \sum_{s=1}^{n} \sum_{t=1}^{d} \mathbb{E}\left[ \frac{\partial f_{i,j}}{\partial x_{st}}(x) \frac{\partial f_{i',j'}}{\partial x_{st}}(\sqrt{t}x+\sqrt{1-t}X) \right] \mathrm{d}t \nonumber \\
    &=\begin{cases}
    \frac{1}{n} \sum_{k=1}^n x_{k,a}^2, & i=i'=j=j'=a \,; \\
    \frac{1}{2n} \sum_{k=1}^n (x_{k,a}^2+x_{k,b}^2), & i=j=a,i'=j'=b \,; \\
    \frac{1}{2n} \sum_{k=1}^{n}( x_{k,a}x_{k,b} ), & \{ i,j \} = \{ a,c \}, \{ i',j' \}=\{ b,c \} \,;  \\
    0 , & \{ i,j \} \cap \{ i',j' \} = \emptyset \,.
    \end{cases}
    \end{align}
    Then by a straightforward computation we have that when $i=i'=j=j'=a$
    \begin{align*}
        \mathbb{E}\Big[ \Big( \mathbf{1}_{(i,j)=(i',j')} - T_{(i,j);(i',j')}(X) \Big)^2 \Big] = \mathbb{E} \Big[ \Big(1-\frac{1}{n} \sum_{k=1}^{n} X_{k,a}^2 \Big)^2 \Big] = \frac{2}{n} \,.
    \end{align*}
    Similarly we have that when $i=j=a,i'=j'=b$
    \begin{align*}
        \mathbb{E}\Big[ \Big( \mathbf{1}_{(i,j)=(i',j')} - T_{(i,j);(i',j')}(X) \Big)^2 \Big] = \mathbb{E} \Big[ \Big(1-\frac{1}{2n} \sum_{k=1}^{n} (X_{k,a}^2+X_{k,b}^2) \Big)^2 \Big] = \frac{1}{n} \,,
    \end{align*}
    and when $\{i,j\}=\{a,c\},\{i',j'\}=\{b,c\}$
    \begin{align*}
        \mathbb{E}\Big[ \Big( \mathbf{1}_{(i,j)=(i',j')} - T_{(i,j);(i',j')}(X) \Big)^2 \Big] = \mathbb{E} \Big[ \Big(\frac{1}{n} \sum_{k=1}^{n} X_{k,a}X_{k,b} \Big)^2 \Big] = \frac{1}{n} \,,
    \end{align*}
    and finally when $\{i,j\}\cap\{i',j'\}=\emptyset$
    \begin{align*}
        \mathbb{E}\Big[ \Big( \mathbf{1}_{(i,j)=(i',j')} - T_{(i,j);(i',j')}(X) \Big)^2 \Big] =0 \,.
    \end{align*}
    Plugging these expressions into \eqref{eq-bound-W-1-dist} yields that $W_1( \frac{1}{\sqrt{2n}} \mathbf{X}^{\mathrm{up}}, \frac{1}{\sqrt{2}} \mathsf{G}^{\mathrm{up}} ) \leq \frac{d}{\sqrt{n}}$, thereby leading to the desired result (2). 
\end{proof}

For any $\Psi_0 \in \mathcal S_d$, define $\mathcal E_1(\Psi_0)$ to be the intersection of the following two events:
    \begin{align}
        & \|U\Psi_0-V\|_{\operatorname{F}} \leq 3 d^3 \widehat{\sigma} \,,  \label{eq-def-good-event-0-III}\\
        &\|\operatorname I_d-nQ \Lambda^{-1} \Psi_0 \Sigma^{-1} R^\top \|_{\operatorname{op}} \leq 2d^3 \widehat{\sigma}  \,.\label{eq-def-good-event-0-IV}
    \end{align}
    In addition, define the event $\mathcal E_\star$ to be the intersection of the following events:
    \begin{align}
        &\sqrt{n}-d\widehat{\omega}_n \leq \min_{1\leq i\leq d} \Lambda_{i,i} \leq \max_{1\leq i\leq d} \Lambda_{i,i} \leq \sqrt{n}+d\widehat{\omega}_n \,,\label{eq-eigen-Lambda}\\
        &\sqrt{(1+\sigma^2)n}-d\widehat{\omega}_n \leq \min_{1\leq i\leq d} \Sigma_{i,i} \leq \max_{1\leq i\leq d} \Sigma_{i,i} \leq \sqrt{(1+\sigma^2)n}+d\widehat{\omega}_n \,,\label{eq-eigen-Sigma}\\
        &\| X \|_{\operatorname{F}}, \| Z \|_{\operatorname{F}} \leq 2\sqrt{nd} \text{ and } \| Y \|_{\operatorname{F}} \leq 2 \sqrt{(1+\sigma^2)nd} \leq 4 \sqrt{nd}\,.\label{eq-Fnorms-X-Y}
    \end{align}
\begin{lemma}{\label{lem-Estar}}
    We have $\Pb(\mathcal E_\star)=1-o(1)$.
\end{lemma}
\begin{proof}
    We need to prove that \eqref{eq-eigen-Lambda}, \eqref{eq-eigen-Sigma} and \eqref{eq-Fnorms-X-Y} hold with probability $1-o(1)$. For \eqref{eq-eigen-Lambda} and \eqref{eq-eigen-Sigma}, it follows from \cite[Theorem 4.6.1]{Vershynin18}. For \eqref{eq-Fnorms-X-Y}, we apply Bernstein's inequality \cite{Bernstein1924} and get that
    \begin{align*}
        &\Pb\left( \big| \|X\|_{\Fop}^2-nd \big|\geq 3nd \right) \leq \exp(-cnd) \,, \\
        &\Pb\left( \big| \|Y\|_{\Fop}^2-(1+\sigma^2)nd \big| \geq 3(1+\sigma^2) nd \right) \leq \exp(-cnd) \,.
    \end{align*}
    Altogether, we complete the proof of the lemma.
\end{proof}
   \begin{lemma}\label{lem-Estar-to-E1Psi_0}
       Assume that $\mathcal E_\star$ holds. Then for any $\Psi_0$ such that $\|Q\Psi_0-R\|_{\Fop}\leq d^3\widehat\sigma $, we have that $\mathcal E_1(\Psi_0)$ holds.
   \end{lemma}
\begin{proof}
To see that \eqref{eq-def-good-event-0-III} holds, recalling that $U=XQ \Lambda^{-1} $ and $V=Y R\Sigma^{-1}$, we have (also recalling $\| \mathsf{AO} \|_{\operatorname{F}} = \| \mathsf{A} \|_{\operatorname{F}}$ for any orthogonal matrix $\mathsf{O}$)
\begin{align}
    & \| U \Psi_0 - V \|_{\operatorname{F}} = \| XQ \Lambda^{-1} \Psi_0 - YR \Sigma^{-1}  \|_{\operatorname{F}} \nonumber \\
    \leq \ & \| X Q (\Lambda^{-1}-\frac{1}{\sqrt{n}} \mathrm{I}_d) \Psi_0 \|_{\operatorname{F}} + \| YR (\Sigma^{-1}-\frac{1}{\sqrt{n}}\mathrm{I}_d) \|_{\operatorname{F}} + \frac{1}{\sqrt{n}} \| XQ \Psi_0 - YR \|_{\operatorname{F}} \nonumber \\
    \leq \ & \| X \|_{\operatorname{F}} \| (\Lambda^{-1}-\frac{1}{\sqrt{n}} \mathrm{I}_d) \|_{\operatorname{op}} + \| Y \|_{\operatorname{F}} \| (\Sigma^{-1}-\frac{1}{\sqrt{n}}\mathrm{I}_d) \|_{\operatorname{op}} + \frac{1}{\sqrt{n}} \| XQ \Psi_0 - YR \|_{\operatorname{F}} \,, \label{eq-matrix-devide-3-parts}
\end{align}
where the first inequality follows from the triangle inequality and the second inequality follows from the fact that $\| \mathsf{AB} \|_{\operatorname{F}} \leq \| \mathsf{A} \|_{\operatorname{F}} \| \mathsf{B} \|_{\operatorname{op}}$.
By \eqref{eq-eigen-Lambda} and \eqref{eq-eigen-Sigma} we have that
\begin{equation}\label{eq-est-singular-value-Lambda-Sigma}
\begin{split}
    &\| \Lambda - \sqrt{n} \operatorname{I}_d \|_{\operatorname{op}}\leq d \widehat{\omega}_n \mbox{ and } \| \Sigma - \sqrt{(1+\sigma^2)n} \operatorname{I}_d \|_{\operatorname{op}} \leq d\widehat{\omega}_n \,;\\
    & \| \Lambda^{-1} - \frac{1}{\sqrt{n}} \operatorname{I}_d \|_{\operatorname{op}} \leq \frac{2d\widehat{\omega}_n}{n} \mbox{ and } \| \Sigma^{-1} - \frac{1}{\sqrt{n}} \operatorname{I}_d \|_{\operatorname{op}} \leq \frac{2\sigma^2\sqrt{n}+2d\widehat{\omega}_n}{n} \,. 
\end{split}
\end{equation}
Combined with \eqref{eq-Fnorms-X-Y}, it yields that
\begin{equation}{\label{eq-bound-matrix-part-1,2}}
    \begin{aligned}
        &\| X \|_{\operatorname{F}} \| \Lambda^{-1}-\frac{1}{\sqrt{n}} \mathrm{I}_d \|_{\operatorname{op}} \leq \frac{ 4d\sqrt{d} \widehat{\omega}_n }{ \sqrt{n} } \overset{\eqref{eq-def-widehat-omega}}{\leq} 0.1d^3 \widehat\sigma \,; \\
        &\| Y \|_{\operatorname{F}} \| (\Sigma^{-1}-\frac{1}{\sqrt{n}}\mathrm{I}_d) \|_{\operatorname{op}} \leq \frac{ 4d\widehat{\omega}_n }{ \sqrt{n} } + 8\sigma^2 \sqrt{d} \overset{\eqref{eq-def-widehat-omega}}{\leq} 0.1d^3 \widehat\sigma \,.
    \end{aligned}
\end{equation}
Also, from \eqref{eq-eigen-Lambda} we have that $\| X \|_{\operatorname{op}} \leq 2\sqrt{n}$ with probability $1-o(1)$, and thus
\begin{align}
    \frac{1}{\sqrt{n}} \| X Q \Psi_0 - Y R \|_{\operatorname{F}} &= \frac{1}{\sqrt{n}} \| X Q \Psi_0 - (X+\sigma Z) R \|_{\operatorname{F}} \nonumber \\
    &\leq \frac{\sigma}{\sqrt{n}} \| Z R \|_{\operatorname{F}} + \frac{1}{\sqrt{n}} \| X (Q \Psi_0-R) \|_{\operatorname{F}} \nonumber \\
    &\leq \frac{\sigma}{\sqrt{n}} \| Z \|_{\operatorname{F}} + \frac{1}{\sqrt{n}} \| X \|_{\operatorname{op}} \| (Q \Psi_0-R) \|_{\operatorname{F}} \nonumber \\
    &\overset{\eqref{eq-Fnorms-X-Y}}{\leq} \sigma\sqrt{d}\widehat\omega_n^{0.1} + 2 d^3 \widehat\sigma \overset{\eqref{eq-def-widehat-omega}}{\leq} 2.5d^3 \widehat\sigma  \,. \label{eq-bound-matrix-part-3} 
\end{align}
Plugging \eqref{eq-bound-matrix-part-1,2} and \eqref{eq-bound-matrix-part-3} into \eqref{eq-matrix-devide-3-parts} yields \eqref{eq-def-good-event-0-III}.
Finally, for \eqref{eq-def-good-event-0-IV} we have 
\begin{align}
    & \big\| \operatorname I_d-nQ \Lambda^{-1} \Psi_0 \Sigma^{-1} R^\top \big\|_{\operatorname{op}} \nonumber \\
    \leq \ & \big\| \operatorname{I}_d-Q \Psi_0 R^\top \big\|_{\operatorname{op}} + \big\| \sqrt{n} Q (\Lambda^{-1}-\frac{1}{\sqrt{n}} \mathrm{I}_d) \Psi_0  R^{\top} \big\|_{\op} + \big\| \sqrt{n} Q \Lambda^{-1} \Psi_0 ( \Sigma^{-1} - \frac{1}{\sqrt{n}} \mathrm{I}_d ) R^{\top} \big\|_{\operatorname{op}} \nonumber \\
    \leq \ & \big\| R - Q \Psi_0 \big\|_{\operatorname{op}} + \sqrt{n} \big\| \Lambda^{-1} - \frac{1}{\sqrt{n}} \operatorname{I}_d \big\|_{\operatorname{op}} + \sqrt{n} \big\| \Lambda^{-1} \big\|_{\op} \big\| \Sigma^{-1} - \frac{1}{\sqrt{n}} \operatorname{I}_d \big\|_{\operatorname{op}}  \,, \label{eq-goal-op-norm-2.10}
\end{align}
where the first inequality follows from the triangle inequality and the second inequality follows from the orthogonal invariance of operator norm and the fact that $\| \mathsf{AB} \|_{\op} \leq \| \mathsf{A} \|_{\op}\| \mathsf{B} \|_{\op}$. Combining \eqref{eq-est-singular-value-Lambda-Sigma} and the fact that $\| \Lambda^{-1} \|_{\op} \leq \frac{2}{\sqrt{n}}$ (from \eqref{eq-eigen-Lambda}), we have 
\begin{align*}
    \eqref{eq-goal-op-norm-2.10} \leq d^3 \widehat\sigma +(2\sigma^2+n^{-0.49}) \overset{\eqref{eq-def-widehat-omega}}{\leq} 2d^3 \widehat{\sigma} \,,
\end{align*}
completing the proof of the lemma.
\end{proof}
We write $$\mathcal E_1 = \bigcup_{\Psi_0\in\mathcal{S}_d} \left(\mathcal E_1(\Psi_0) \cap \{\|Q\Psi_0-R\|_{\Fop}\leq d^3\widehat\sigma\}\right)\,.$$
\begin{cor}\label{cor-diff-orthogonal}
    We have $\Pb(\mathcal E_1) = 1-o(1)$.
\end{cor}
\begin{proof}[Proof of Corollary~\ref{cor-diff-orthogonal}]
    This follows immediately by combining Proposition~\ref{prop-diff-orthogonal-matrix}, Lemma~\ref{lem-Estar} and Lemma~\ref{lem-Estar-to-E1Psi_0}.
\end{proof}
    

Now we shift our focus to the distance model. Recall \eqref{eq-goal-distance-model}, where we have
\begin{equation*}
    \widetilde A=\widetilde X\widetilde X^\top \mbox{ and }
    \widetilde B=\widetilde Y\widetilde Y^\top \,,
\end{equation*}
where $\widetilde X=(\operatorname{I}_n-\mathbf F)X$ and $\widetilde Y=(\operatorname{I}_n-\mathbf F)Y$.
In the same manner as in \eqref{eq-def-SVD}, we consider the singular value decompositions of $\widetilde X$ and $\widetilde Y$: 
\begin{equation}{\label{eq-def-SVD-distance}}
    \mathsf{SVD}(\widetilde X)=\widetilde U\widetilde \Lambda \widetilde Q^\top  \mbox{ and }  \mathsf{SVD}(\widetilde Y)=\widetilde V\widetilde \Sigma \widetilde R^\top\,.
\end{equation}

Similar to the dot-product model, the crucial input is to show that $\widetilde R$ and $\widetilde Q$ are ``close'' up to a multiplicative matrix in $\mathcal{S}_d$, as incorporated in the next proposition.
\begin{proposition}\label{prop-diff-othogonal-matrix-distance}
    For $\widetilde Q$ and $\widetilde R$ and any $\omega_n\to\infty$, with probability $1-o(1)$ there exists $\widetilde{\Psi}_0 \in \mathcal{S}_d$ such that
    \begin{equation}{\label{eq-closeness-tilde-Q-R}}
        \|\widetilde{Q} \widetilde{\Psi}_0-\widetilde{R} \|_{\operatorname{F}} \leq 2d^3 \widehat{\sigma} \,.
    \end{equation}
\end{proposition}
\begin{proof}
    Consider the centered matrices $\widetilde{\mathbf{X}}:=\widetilde X^\top \widetilde X-n\operatorname{I}_d$ and $\widetilde{\mathbf{Y}}:=\widetilde Y^\top\widetilde Y-n(1+\sigma^2)\operatorname{I}_d$. Let $\widetilde q_i,\widetilde r_i$ be the eigenvectors corresponding to the $i$-th largest eigenvalue of $\widetilde{\mathbf{X}}$ and $\widetilde{\mathbf{Y}}$, respectively. By the Markov's inequality, we have that with probability $1-o(1)$
    \begin{equation}\label{eq-diff-X-Xtilde}
        \|\widetilde{\mathbf{X}}-\mathbf X\|_{\operatorname{op}}=\|X^\top \mathbf F X\|_{\operatorname{op}} = \frac{1}{n}\| X^\top \mathbf{1} \|_2^2 \leq d \widehat{\omega}_n^{0.1} \,.
    \end{equation}
    Moreover by \eqref{eq-goal-spectral-gap} and by applying the Davis-Kahan theorem to $(\widetilde{\mathbf X}, \mathbf X)$, we have that with probability $1-o(1)$ (again recall $d=O(\log n)$)
    \begin{equation}
        \sin{\angle(\widetilde q_i,q_i)} \leq \frac{2d^2 \widehat{\omega}_n}{\sqrt{n}} \overset{\eqref{eq-def-widehat-omega}}{\leq} \frac{1}{ n^{0.495}} \mbox{ for } 1 \leq i \leq d \,.
    \end{equation}
    Thus, similar to how Proposition~\ref{prop-diff-orthogonal-matrix} was derived from Claim~\ref{lem-bound-angle}, we can show that with probability $1-o(1)$ there exists $\Psi_1 \in \mathcal{S}_d$ such that
    \begin{equation}\label{eq-diff-QQtilde}
        \|\widetilde Q-Q\Psi_1\|_{\operatorname{F}} \leq n^{-0.49}\,.
    \end{equation}
    Similarly with probability $1-o(1)$, there exists $\Psi_2\in\mathcal{S}_d$ such that 
    \begin{align}\label{eq-diff-RRtilde}
        \|\widetilde R-R\Psi_2\|_{\operatorname{F}}\leq n^{-0.49}\,.
    \end{align}
    Combining \eqref{eq-diff-QQtilde}, \eqref{eq-diff-RRtilde} with Proposition~\ref{prop-diff-orthogonal-matrix} yields that by defining $\widetilde{\Psi}_0 = \Psi_1 \Psi_0 \Psi_2$, with probability $1-o(1)$ we have
    \begin{align}
        &\| \widetilde{Q} \widetilde{\Psi}_0 - \widetilde{R} \|_{\operatorname{F}} = \| \widetilde{Q} \Psi_1 \Psi_0 \Psi_2 - \widetilde{R} \|_{\operatorname{F}} \nonumber \\
        \leq \ &\| (\widetilde{Q} \Psi_1 - Q) \Psi_0 \Psi_2 \|_{\operatorname{F}} + \| (Q \Psi_0 - R) \Psi_2 \|_{\operatorname{F}} + \|  \widetilde{R} - R \Psi_2 \|_{\operatorname{F}} \nonumber \\
        = \ &\| \widetilde{Q} \Psi_1 - Q \|_{\operatorname{F}} + \| Q \Psi_0 - R \|_{\operatorname{F}} + \|  \widetilde{R} - R \Psi_2 \|_{\operatorname{F}} \leq 2n^{-0.49} + d^3 \widehat{\sigma} \overset{\eqref{eq-def-widehat-omega}}{\leq} 2d^3 \widehat{\sigma} \,. \qedhere
    \end{align}
\end{proof}

For any $\widetilde\Psi_0 \in \mathcal S_d$, define $\widetilde{\mathcal E_1}(\widetilde{\Psi}_0)$ to be the intersection of the following two events:
\begin{align}
    &\|\widetilde{U}\widetilde{\Psi}_0-\widetilde{V} \|_{\operatorname{F}} \leq 4d^3 \widehat{\sigma} \,.\label{eq-diff-UV-dist}\\
        &\| \operatorname I_d-n \widetilde{Q} \widetilde\Lambda^{-1} \widetilde{\Psi}_0 \widetilde\Sigma^{-1} \widetilde{R}^\top \|_{\operatorname{op}} \leq 3d^3 \widehat{\sigma} \,.\label{eq-Id-diff-dist}
\end{align}
In addition, define the event $\widetilde{\mathcal E}_\star$ to be the intersection of the following events:
\begin{align}
    &\|X-\widetilde X\|_{\Fop}^2\leq d\widehat\omega_n\,,\label{eq-Fnorm-FX}\\
    &\sqrt{n}- 2d\widehat{\omega}_n \leq \min_{1\leq i\leq d} \widetilde\Lambda_{i,i} \leq \max_{1\leq i\leq d} \widetilde\Lambda_{i,i} \leq \sqrt{n}+ 2d\widehat{\omega}_n \,,\label{eq-singular-tildeX}\\
    &\sqrt{n(1+\sigma^2)}-2d\widehat{\omega}_n \leq \min_{1\leq i\leq d} \widetilde\Sigma_{i,i} \leq \max_{1\leq i\leq d} \widetilde\Sigma_{i,i} \leq \sqrt{n(1+\sigma^2)}+2d\widehat{\omega}_n \,.{\label{eq-singular-tildeY}} 
\end{align}
\begin{lemma}\label{lem-Etildestar}
    We have $\Pb(\widetilde{\mathcal{E}}^\star)=1-o(1)$.
\end{lemma}
\begin{proof}[Proof of Lemma~\ref{lem-Etildestar}]
    For \eqref{eq-Fnorm-FX}, we have that with probability $1-o(1)$ (recall that $\widetilde{X}=(\mathrm{I}_n - \mathbf{F})X$)
    \begin{equation*}
        \big\| X-\widetilde{X} \big\|_{\operatorname{F}}^2 =\big\|\mathbf{F}X\big\|^2_{\operatorname{F}} =\frac{1}{n} \sum_{j=1}^d \Big(\sum_{k=1}^n X_{k,j} \Big)^2 \leq d \widehat{\omega}_n \,,
    \end{equation*}
    where the inequality follows from Markov's inequality. In addition, we also have
    \begin{equation*}
        \big\| X^{\top}X - \widetilde{X}^{\top} \widetilde{X} \big\|_{\operatorname{op}} = \big\| (\mathbf{F} X)^{\top} \mathbf{F} X \big\|_{\operatorname{op}} \leq \big\| \mathbf{F}X \big\|^2_{\operatorname{F}} \leq d \widehat{\omega}_n \,.
    \end{equation*}
    By recalling that $\Lambda_{i,i}$ and $\widetilde{\Lambda}_{i,i}$ are singular values of $X$ and $\widetilde{X}$ respectively, we conclude that $|\Lambda_{i,i}-\widetilde{\Lambda}_{i,i}| \leq d \widehat{\omega}_n$. Combined with \eqref{eq-eigen-Lambda}, this yields \eqref{eq-singular-tildeX}. In a similar manner we can derive that \eqref{eq-singular-tildeY} holds with probability $1-o(1)$. This completes the proof of the lemma.
\end{proof}

\begin{lemma}\label{lem-Etildestar-to-Etildepsi}
       Assume that $\widetilde{\mathcal E}_\star$ holds. Then for any $\widetilde\Psi_0$ such that $\|\widetilde Q\widetilde\Psi_0-\widetilde R\|_{\Fop}\leq 2d^3\widehat\sigma $, we have that $\widetilde{\mathcal E}_1(\widetilde\Psi_0)$ holds.
   \end{lemma}

\begin{proof}
    To show \eqref{eq-diff-UV-dist}, by \eqref{eq-singular-tildeX} and \eqref{eq-singular-tildeY}, we can bound $\| \widetilde U \widetilde\Psi_0 - \widetilde V \|_{\operatorname{F}}$ by the triangle inequality
    \begin{align*}
        &\| \widetilde U \widetilde\Psi_0 - \widetilde V \|_{\operatorname{F}} = \| \widetilde X\widetilde Q \widetilde\Lambda^{-1} \widetilde\Psi_0 - \widetilde Y\widetilde R \widetilde\Sigma^{-1}  \|_{\operatorname{F}} \\
        \leq \ & \frac{1}{\sqrt{n}} \| \widetilde X\widetilde Q \widetilde\Psi_0 - \widetilde Y\widetilde R \|_{\operatorname{F}} + \frac{3d\widehat{\omega}_n}{n} \| \widetilde X \widetilde Q \widetilde\Psi_0 \|_{\operatorname{F}} + \frac{3\sigma^2\sqrt{n}+3d\widehat{\omega}_n}{n} \| \widetilde Y \widetilde R \|_{\operatorname{F}} \,.
    \end{align*}
    Following the proof of Lemma~\ref{lem-Estar-to-E1Psi_0} (just the same as how we bounded \eqref{eq-matrix-devide-3-parts}), we can get that the right hand side in the preceding display is bounded by
    $$ (3d^3 \widehat{\sigma} +3n^{-0.49}) +n^{-0.495}+(3\sigma^2\sqrt{d}+n^{-0.495}) \overset{\eqref{eq-def-widehat-omega}}{\leq} 4d^3 \widehat{\sigma} \,, $$
    which yields \eqref{eq-diff-UV-dist}.
    Finally, similar to the proof of Lemma~\ref{lem-Estar-to-E1Psi_0}, we can deduce from the triangle inequality that
    \begin{align}
        &\|\operatorname{I}_d-n \widetilde{Q}  \widetilde{\Lambda}^{-1} \widetilde\Psi_0 \widetilde{\Sigma}^{-1} \widetilde{R}^\top\|_{\op} \nonumber \\
        \leq \ & \big\| \widetilde R - \widetilde Q \widetilde \Psi_0 \big\|_{\operatorname{op}} + \sqrt{n} \big\| \widetilde{\Lambda}^{-1} - \frac{1}{\sqrt{n}} \operatorname{I}_d \big \|_{\operatorname{op}} + \sqrt{n} \big\| \widetilde{\Lambda}^{-1} \big\|_{\op} \big\| \widetilde{\Sigma}^{-1} - \frac{1}{\sqrt{n}} \operatorname{I}_d \big\|_{\operatorname{op}} \,. \label{eq-goal-op-norm-distance-2.25}
    \end{align}
    Following the previous proof as how we bounded \eqref{eq-goal-op-norm-2.10}, we obtain
    \begin{align*}
        \eqref{eq-goal-op-norm-distance-2.25} \leq & 2d^3 \widehat{\sigma} +2n^{-0.49} +\sigma^2 +n^{-0.495} \leq 3d^3 \widehat{\sigma} \,,
    \end{align*}
    completing the proof of the lemma.
\end{proof}
We write $$\widetilde{\mathcal E}_1 =\bigcup_{\widetilde\Psi_0\in\mathcal{S}_d} \left(\widetilde{\mathcal E_1}(\widetilde\Psi_0) \cap \{\|\widetilde Q\widetilde \Psi_0-\widetilde R\|_{\Fop}\leq 2d^3\widehat\sigma\}\right)\,.$$
\begin{cor}\label{cor-diff-UV-distance-model}
    We have $\Pb(\widetilde{\mathcal E}_1)=1-o(1)$.
\end{cor}
\begin{proof}[Proof of Corollary~\ref{cor-diff-UV-distance-model}]
    This follows immediately by combining Proposition~\ref{prop-diff-othogonal-matrix-distance}, Lemma~\ref{lem-Etildestar} and Lemma~\ref{lem-Etildestar-to-Etildepsi}.
\end{proof}
With these results in hand, we are prepared to prove the main theorems in the following subsections.

\subsection{Proofs of Almost Exact Recovery}
In this subsection we prove Theorems~\ref{thm-almost-exact-recovery} and \ref{thm-almost-exact-recovery-distance}.
Without loss of generality, we may assume that $\Pi^*=\operatorname{I}_n$. Recall that we assumed $\sigma=n^{-1/d} d^{-3}/\omega_n$ and also recall the definition of $\widehat{\omega}_n$ in \eqref{eq-def-widehat-omega}. Define 
\begin{equation}{\label{eq-direct-form-epsilon}}
    \epsilon:=\frac{50}{ d \log [( \sigma^{-1}d^{-3} n^{-1/d} \wedge \log n ) ] }=\frac{ 25 }{ d \log (\widehat\omega_n) } \,.
\end{equation}
 Define the event 
\begin{equation}{ \label{eq-def-good-event-1} }
    \mathcal E_2 = \Big\{ \|X\|_{\operatorname{F}}, \|Y\|_{\operatorname{F}}, \|Z\|_{\operatorname{F}} \leq 2\sqrt{nd} \Big\} \,. 
\end{equation}
In addition, we define
\begin{equation}{\label{eq-def-mathfrak-S-n,epsilon}}
    \mathfrak S_{n,\epsilon}:=\{\Pi\in\mathfrak S_n:d_{\operatorname{H}}(\Pi,\operatorname{I}_n)\geq\epsilon n\}\,,
\end{equation}
and then define (recall the definition of $\widehat{\sigma}$ in \eqref{eq-def-widehat-omega})
\begin{equation}{\label{eq-def-good-event-2}}
    \mathcal E_3 = \Big\{ \|\Pi U-U \Psi\|_{\operatorname{F}} \geq 8.5d^3 \widehat{\sigma}  \,, \forall \ \Pi \in \mathfrak{S}_{n,\epsilon}\text{ and } \Psi \in \mathcal{S}_d \Big\} \,.
\end{equation}
\begin{lemma}\label{lem-proximity-U-Z}
    Recall the definition of $U$ in \eqref{eq-def-SVD} and consider an $n\!*\!d$ matrix $\mathsf Z$ where the entries of $\mathsf Z$ are given by i.i.d. standard normal variables. We have that there exists $\mathsf U\overset{d}{=}U$ such that with probability $1-o(1)$
    \begin{equation}
        \big\| \mathsf U-\frac{1}{\sqrt{n}}\mathsf Z \big\|_{\operatorname{F}} \leq 8d^2 n^{-0.49}\,.\notag
    \end{equation}
\end{lemma}
\begin{proof}
    Denote $\mathsf{Z}=[\mathsf{z}_1, \ldots, \mathsf{z}_d]$. By Bernstein's inequality we have  
\begin{equation}\label{eq-append-concentration}
    \begin{split}
    &\Pb\big( |\langle\mathsf{z}_i,\mathsf{z}_j\rangle|\leq n^{0.51}, \forall \ 1\leq i\neq j\leq d \big)=1-o(1) \\
    \mbox{and }&\Pb\big( n/2\leq \|\mathsf{z}_i\|^2\leq 2n, \forall \ 1\leq i\leq d \big)=1-o(1)\,,
    \end{split}
\end{equation}
and thus we may we assume without loss of generality that the events described in \eqref{eq-append-concentration} hold in what follows. Denote $\mathsf{Z} = \mathsf{U} \mathsf{L}$ as the QR-decomposition of $\mathsf{Z}$. We first claim we have $\mathsf{U}\overset{(d)}{=}U$. To this end, define $\mathsf{Z}_{\mathsf{ext}}=[ \mathsf{z}_{d+1}, \ldots, \mathsf{z}_n ]$ to be an $n\!*\!(n-d)$ matrix whose entries are i.i.d.\ standard normal variables that are independent of $\mathsf{Z}$. Consider the QR-decomposition for this $n\!*\!n$ extended matrix as follows:
\begin{equation*}
    \begin{pmatrix} \mathsf{Z} & \mathsf{Z}_{\mathsf{ext}} \end{pmatrix} = 
    \begin{pmatrix}  \mathsf U & \mathsf U_{\mathsf{ext}} \end{pmatrix} 
    \begin{pmatrix}
        \mathsf L & \mathsf L_{\mathsf{ext}^1}\\
        \mathsf 0 & \mathsf L_{\mathsf{ext}^2}
    \end{pmatrix} \,.
\end{equation*}
By \cite[Theorem 3.2]{Stewart80} we have that $(\mathsf U,\mathsf U_{\mathsf{ext}})$ is uniformly distributed on $\mathsf{SO}(n)$, where $\mathsf{SO}(n)$ is the set of $n\!*\!n$ orthogonal matrices. This shows that $\mathsf U\overset{(d)}{=}U$. Now we bound $\| \mathsf U - \frac{1}{\sqrt{n}} \mathsf Z \|_{\Fop}$.
Let $\Bar{\mathsf{Z}}:=\big[ \Bar{\mathsf z}_1,\cdots,\Bar{\mathsf z}_d \big]= \big[\frac{\mathsf z_1}{\Vert \mathsf z_1\Vert_2},\cdots,\frac{\mathsf z_d}{\Vert \mathsf z_d\Vert_2} \big]$. By the triangular inequality, we have
\begin{equation}\label{eq-F-norm-relaxation-1}
    \big\| \mathsf{U} - \frac{1}{\sqrt{n}} \mathsf Z \big\|_{\operatorname{F}} \leq \big\| \mathsf{U} - \Bar{\mathsf Z} \big\|_{\operatorname{F}} + \big\| \Bar{\mathsf Z} - \frac{1}{\sqrt{n}} \mathsf Z \big\|_{\operatorname{F}}\,.
\end{equation}
Standard concentration result yields that with probability $1-o(1)$
\begin{equation}{\label{eq-F-norm-relaxation-2}}
    \big\| \Bar{\mathsf Z} - \frac{1}{\sqrt{n}} \mathsf Z \big\|_{\operatorname{F}}^2 = \sum_{i=1}^{d} \big\| \Bar{\mathsf{z}}_i - \frac{1}{\sqrt{n}} \mathsf{z}_i \big\|^2 = \sum_{i=1}^{d} \Big( 1-\frac{\| \mathsf{z}_i \|}{\sqrt{n}} \Big)^2 \leq \frac{d}{n^{0.99}} \,.
\end{equation}
It remains to bound $\| \mathsf{U} - \Bar{\mathsf Z} \|_{\operatorname{F}}$.
Note that $\mathsf U$ can be constructed by applying Gram–Schmidt process to $\Bar{\mathsf Z}$ as follows: writing $\mathsf{U}=[\mathsf{u}_1,\ldots,\mathsf{u}_d]$, then for $i\geq 0$ we inductively have that
\begin{equation}{\label{eq-def-mathsf-u}}
    \mathsf{u}_{i+1} = \frac{ \Bar{\mathsf z}_{i+1}-\sum_{j=1}^i \langle \Bar{\mathsf z}_{i+1}, \mathsf{u}_j \rangle \mathsf{u}_j }{ \| \Bar{\mathsf z}_{i+1} - \sum_{j=1}^i \langle \Bar{\mathsf z}_{i+1},\mathsf{u}_j \rangle \mathsf{u}_j\| }\,.
\end{equation}
In light of the preceding inductive relation, we next prove by induction on $k$ that the following holds for all $1\leq k<d$ 
\begin{equation} \label{eq-induc-E2} 
|\langle \mathsf{u}_k, \Bar{\mathsf z}_l \rangle| \leq 2n^{-0.49}  \mbox{ for all } l \in (k, d].  
\end{equation}
The case of $k=1$ is evident based on our assumption from \eqref{eq-append-concentration}. Suppose that \eqref{eq-induc-E2} holds for some $k<d$ and we wish to prove by induction that \eqref{eq-induc-E2} holds also for $k+1$. To this end, we consider each $l\in(k+1,d]$ and we get from \eqref{eq-def-mathsf-u} and our assumption $d=O(\log n)$ that
\begin{align}
        |\langle \mathsf{u}_{k+1}, \Bar{\mathsf z}_l \rangle| &= \frac{ \big| \big\langle \Bar{\mathsf z}_l,\Bar{\mathsf z}_{k+1}-\sum_{j=1}^{k}\langle \Bar{\mathsf z}_{k+1},\mathsf u_j\rangle \mathsf u_j \big\rangle \big| }{ \|\Bar{\mathsf z}_{k+1}-\sum_{j=1}^{k}\langle \Bar{\mathsf z}_{k+1},\mathsf u_j\rangle \mathsf u_j\| } \leq \frac{ | \langle \Bar{\mathsf z}_l,\Bar{\mathsf z}_{k+1} \rangle | + \sum_{j=1}^{k} | \langle \Bar{\mathsf z}_{k+1},\mathsf u_j \rangle | \cdot | \langle \Bar{\mathsf z}_l, \mathsf u_j \rangle | }{ \|\Bar{\mathsf z}_{k+1}\| - \sum_{j=1}^{k} |\langle \Bar{\mathsf z}_{k+1},\mathsf u_j\rangle| \cdot \| \mathsf u_j\| } \nonumber \\
        &\leq \frac{n^{-0.49}+k(2n^{-0.49})^2}{1-2kn^{-0.49}}\leq 2n^{-0.49} \,,
\end{align}
where the second inequality follows from \eqref{eq-append-concentration} and the induction hypothesis.
This shows that for all $1\leq k\leq d-1$, \eqref{eq-induc-E2} is true. Provided with this fact, we have $\mathsf{u}_1=\Bar{\mathsf{z}}_1$ and for $1\leq i<d$ (recall that $\| \mathsf{u}_i \|=\| \Bar{\mathsf{z}}_i \|=1$)
\begin{align*}
    \| \mathsf{u}_{i+1} - \Bar{\mathsf{z}}_{i+1} \| & \overset{\eqref{eq-def-mathsf-u}}{\leq} \Big| 1-\frac{ 1 }{ \| \Bar{\mathsf z}_{i+1} - \sum_{j=1}^i \langle \Bar{\mathsf z}_{i+1},\mathsf{u}_j \rangle \mathsf{u}_j\| } \Big| + \sum_{j=1}^{i} \frac{ |\langle \Bar{\mathsf z}_{i+1}, \mathsf{u}_j \rangle| }{ \| \Bar{\mathsf z}_{i+1} - \sum_{j=1}^i \langle \Bar{\mathsf z}_{i+1},\mathsf{u}_j \rangle \mathsf{u}_j\| }  \\
    &\leq \Big| 1- \frac{ 1 }{ 1 - \sum_{j=1}^{i} n^{-0.49} } \Big| + \sum_{j=1}^{i} \frac{ n^{-0.49} }{ 1- \sum_{j=1}^{i} n^{-0.49} } \leq 4d n^{-0.49} \,.
\end{align*}
Thus,
\begin{equation}{ \label{eq-F-norm-relaxation-3} }
    \| \mathsf{U} - \Bar{\mathsf{Z}} \|_{\operatorname{F}}^2 = \sum_{i=1}^{d} \| \mathsf{u}_i - \Bar{\mathsf{z}}_i \|^2 \leq 16d^3 n^{-0.98}\,.
\end{equation}
Plugging \eqref{eq-F-norm-relaxation-2} and \eqref{eq-F-norm-relaxation-3} into \eqref{eq-F-norm-relaxation-1} yields Lemma~\ref{lem-proximity-U-Z}.
\end{proof}
\begin{lemma}\label{lem-typical-E3}
    We have that $\Pb(\mathcal{E}_3)=1-o(1)$.
\end{lemma}
\begin{proof}
    Provided with Lemma~\ref{lem-proximity-U-Z}, it remains to prove (recall that we have assumed $d \geq 5$ and thus we have  
    $n^{-0.49} \ll \widehat{\sigma}$) 
\begin{equation}\label{eq-goal2-E2}
    \Pb \Big( \exists\,\Pi \in \mathfrak S_{n,\epsilon}, \Psi \in \mathcal{S}_d \,\mbox{ such that} \, \| \Pi \mathsf{Z} \Psi-\mathsf{Z} \|_{\operatorname{F}} \leq 9 d^3 \sqrt{n} \widehat{\sigma} \Big) =o(1) \,.
\end{equation}
By a union bound and Chebyshev's inequality, the left hand side of \eqref{eq-goal2-E2} is upper-bounded by
\begin{equation}\label{eq-union-E2}
    \sum_{ \Pi\in\mathfrak S_{n,\epsilon} } \sum_{ \Psi \in \mathcal{S}_d } e^{81t nd^6 \widehat{\sigma}^2} \mathbb{E} \Big[ e^{-t \| \mathsf{Z}-\Pi\mathsf{Z}\Psi \|_{\operatorname{F} }^2 } \Big] \,.
\end{equation}
Taking $t=1/(32{\sigma}_0^2)$ where ${\sigma}_0=d^3 \widehat{\sigma}$, it suffices for us to prove that
\begin{equation}\label{eq-final-goal-high-noise-E2}
    \sum_{\Pi\in\mathfrak S_{n,\epsilon}} \sum_{\Psi\in \mathcal{S}_d} \mathbb{E} \Big[ e^{-\|X-\Pi X\Psi\|_{\operatorname{F} }^2/ (32{\sigma}_0^2) } \Big] = o\Big(e^{-81n/32} \Big)\,.
\end{equation} 
By \cite[Lemma~4]{WWXY22+}, we have 
$$ \mathbb{E} \Big[ e^{-\|X-\Pi X\Psi\|_{\operatorname{F}}^2 /(32{\sigma}_0^2) } \Big] = \prod_{k=1}^{n} [a_k(\Psi)]^{n_k} \,, $$
where $n_k$ is the number of $k$-cycles in the permutation $\Pi$, and for $\Psi=\operatorname{diag}(\psi_1,\ldots,\psi_d)$ the term $a_k(\Psi)$ is defined to be 
$$ a_k(\Psi)=(4\sigma_0)^{kd} \prod_{l=1}^d \Bigg[ \Big( \sqrt{1+4\sigma_0^2}+2\sigma_0 \Big)^{2k} + \Big( \sqrt{1+4\sigma_0^2}-2\sigma_0 \Big)^{2k} -2\psi_l^k \Bigg]^{-1/2} \,. $$
A direct calculation yields the following estimation on $a_k(\Psi)$: 
\begin{align}
    a_1(\Psi) &= (4\sigma_0)^d \prod_{l=1}^d (2-2\psi_l+16 \sigma_0^2)^{-1/2} \leq 1 \,, \label{eq-bound-a-1-Psi} \\
    a_k(\Psi) &\leq a_k(\mathrm{I}_d) = (4\sigma_0)^{kd} \Bigg[ \Big( \sqrt{1+4\sigma_0^2}+2\sigma_0 \Big)^{k} - \Big( \sqrt{1+4\sigma_0^2}+2\sigma_0 \Big)^{-k} \Bigg]^{-d} \nonumber \\
    & \leq (4\sigma_0)^{kd} \Big[ (1+4\sigma_0)^{k} - (1+4\sigma_0)^{-k} \Big]^{-d} \leq (4\sigma_0)^{(k-1)d}  \,. \label{eq-bound-a-k-Psi}
\end{align}
Applying \eqref{eq-bound-a-1-Psi} and \eqref{eq-bound-a-k-Psi}, we obtain 
\begin{align*}
    \mathbb{E} \Big[ e^{-\|X-\Pi X \Psi \|_{\operatorname{F}}^2 /(32\sigma_0^2) } \Big] &\leq  \prod_{j=2}^{n} \Big( (4\sigma_0)^{d(j-1)} \Big)^{n_j} 
    = (4\sigma_0)^{ \sum_{j=2}^{n} d(j-1)n_j } = (4\sigma_0)^{ d(n - \sum_{j=1}^{n} n_j) }  \,,
\end{align*}
where the last equality comes from 
$$ \sum_{j=2}^{n} (j-1)n_j = \sum_{j=1}^{n} jn_j - \sum_{j=1}^{n} n_j = n - \sum_{j=1}^{n} n_j \,. $$ 
We denote by $m=n-n_1$ be the number of non-fixed points of $\pi$. Let $\pi^\star$ be the restriction of $\pi$ on non-fixed points, and thus we see that $\pi^\star$ is a derangement on the set of non-fixed points. Let $\mathfrak c(\pi)$ be the number of cycles in $\pi$, we immediately have 
$$ \mathfrak{c}(\pi^\star) = \mathfrak{c}(\pi) - n_1 = \sum_{j=2}^{n} n_j\,, $$
and thus $n-\sum_{j=1}^{n} n_j=n-n_1-\mathfrak{c}(\pi^\star) =m-\mathfrak{c}(\pi^\star)$. So the left hand side of \eqref{eq-final-goal-high-noise-E2} is bounded by
\begin{equation}\label{eq-union-E2-upboundsigma}
    \sum_{m=\epsilon n}^n 2^d \binom{n}{m} \sum_{\pi^\star \mathsf{derangement}} (4\sigma_0)^{d(m-\mathfrak c(\pi^\star))} = \sum_{m=\epsilon n}^n 2^d \binom{n}{m}(4\sigma_0)^{dm} \sum_{\pi^\star\mathsf{derangement}}(4\sigma_0)^{-d\mathfrak c(\pi^\star)} \,.
\end{equation}
Applying \cite[Equation~(49)]{WWXY22+} with $L=(4\sigma_0)^{-d}$, we derive that
\begin{equation}
    \sum_{\pi^\star\mathsf{derangement}} (4\sigma_0)^{-d\mathfrak c(\pi^\star)}\leq m! \cdot \Bigg( \frac{16\cdot(4\sigma_0)^{-d}}{m} \Bigg)^{m/2} \,.
\end{equation}
Plugging the above expression into \eqref{eq-union-E2-upboundsigma} yields that the left hand side of \eqref{eq-final-goal-high-noise-E2} is bounded by
\begin{align*}
    &\sum_{m=\epsilon n}^n 2^d \binom{n}{m} (4\sigma_0)^{md} \cdot m! \Bigg( \frac{16\cdot(4\sigma_0)^{-d}}{m} \Bigg)^{m/2} \leq \sum_{m=\epsilon n}^n 2^d n^m \Bigg( \frac{16\cdot(4\sigma_0)^{d}}{m} \Bigg)^{m/2} \\
    = &\sum_{m=\varepsilon n}^n \exp \Bigg( m\log n+  \frac{md}{2}\log (4\sigma_0)- \frac{m}{2} \log m+ d\log 2 +\frac{m}{2}\log 16 \Bigg) \\
    \leq &\sum_{m=\varepsilon n}^n \exp \Bigg( m \Big( \log n - \frac{1}{2} \log m - \frac{1}{2} \log n - \frac{d}{2} \log \frac{\widehat\omega_n}{4} + \frac{1}{2}\log 16 \Big) +d\log 2\Bigg) \\
    \leq &\sum_{m=\varepsilon n}^n \exp \Bigg( m \Big( \log (\epsilon^{-1}) - \frac{d}{4} \log \widehat\omega_n \Big) +d\log 2\Bigg) \\
    \overset{\eqref{eq-direct-form-epsilon}}{\leq}    &\sum_{m=\varepsilon n}^n \exp \Big( -\frac{1}{8} md \log (\widehat\omega_n) \Big) \leq n \exp \Big( -\frac{1}{8} \epsilon nd \log (\widehat\omega_n) \Big) \overset{\eqref{eq-direct-form-epsilon}}{=} o(e^{-81n/32}) \,,
\end{align*}
where the second inequality follows from $\sigma_0=d^3 \widehat{\sigma}= n^{-1/d} /{\widehat\omega_n}$ and $d=O(\log n)$, and the third inequality follows from $m \geq \epsilon n$. Thus we conclude that \eqref{eq-final-goal-high-noise-E2} holds, completing the proof of \eqref{eq-goal2-E2} (and so does the proof of the whole lemma).
\end{proof}
\begin{lemma}{\label{lem-total-good-event-almost-exact-recovery}}
    Let $\mathcal E_{\diamond}=\mathcal E_1 \cap \mathcal E_2 \cap \mathcal E_3$. We then have $\mathbb{P}(\mathcal E_{\diamond})=1-o(1)$.
\end{lemma}
\begin{proof}
The proof of $\Pb(\mathcal E_2)=1-o(1)$ is standard and follows from Bernstein's inequality. Then Lemma~\ref{lem-total-good-event-almost-exact-recovery} follows by combining the fact that $\Pb(\mathcal E_2)=1-o(1)$, Corollary~\ref{cor-diff-orthogonal} and Lemma~\ref{lem-typical-E3}.
\end{proof}

Now we are prepared to prove Theorem~\ref{thm-almost-exact-recovery}.

\begin{proof}[Proof of Theorem~\ref{thm-almost-exact-recovery}.]
    It suffices to show that $\mathcal E_{\diamond}$ implies (recall the definition of $\Psi_0$ in $\mathcal{E}_0$)
    \begin{equation}{ \label{eq-partial-recovery-goal-1} }
        \Big\langle \Pi U \Psi, V \Big\rangle \leq \Big\langle U \Psi_0, V \Big\rangle \,, \forall \,  \Pi\in\mathfrak S_{n,\epsilon}, \Psi \in \mathcal{S}_d \,.
    \end{equation}
    In fact, using $\| M \Phi \|_{\operatorname{F}} = \| M \|_{\operatorname{F}}$ for all $M \in \mathbb R^{n*d}$ and $\Phi \in \mathcal{S}_d$, under $\mathcal E_3$ we have
    \begin{equation}{ \label{eq-partial-recovery-relaxation-1} }
        \big\| \Pi U \Psi - U \Psi_0 \big\|_{\operatorname{F}} \geq 8.5 d^3 \widehat{\sigma} \,, \forall \,  \Pi\in\mathfrak S_{n,\epsilon}, \Psi \in \mathcal{S}_d \,.
    \end{equation} 
    Note that \eqref{eq-def-good-event-0-III} in the definition of $\mathcal E_1$ implies that $\| U \Psi_0 - V \|_{\operatorname{F}} \leq 3d^3 \widehat\sigma$. Using the triangle inequality that $\| \Pi U \Psi - V \|_{\operatorname{F}} \geq \| \Pi U \Psi - U \Psi_0 \|_{\operatorname{F}} - \| U \Psi_0 - V \|_{\operatorname{F}}$ we have
    \begin{equation}{ \label{eq-partial-recovery-relaxation-2} }
        \big\| \Pi U \Psi - V \big\|_{\operatorname{F}} \geq 5 d^3 \widehat{\sigma} \,, \forall\,  \Pi\in\mathfrak S_{n,\epsilon}, \Psi \in \mathcal{S}_d \,.
    \end{equation}
    Applying \eqref{eq-def-good-event-0-III} again, we have
    \begin{equation}{ \label{eq-partial-recovery-relaxation-3} }
        \big\| \Pi U \Psi - V \big\|_{\operatorname{F}}^2 \geq \big\| U \Psi_0 - V \big\|_{\operatorname{F}}^2, \forall\, \Pi\in\mathfrak S_{n,\epsilon} \,, \Psi \in \mathcal{S}_d \,.
    \end{equation}
    This implies \eqref{eq-partial-recovery-goal-1} immediately since $\|U\Psi_0\|_{\operatorname{F}} = \|\Pi U\Psi\|_{\operatorname{F}}$ for all $\Pi \in \mathfrak{S}_n$ and $\Psi\in \mathcal{S}_d$.
\end{proof}

Now we conclude the proof of almost exact recovery for the distance model. 
The next lemma characterizes the difference between $\widetilde U$ and $U$ (respectively, $\widetilde V$ and $V$) up to an element in $\mathcal S_d$.
\begin{lemma}\label{lem-diff-UUtilde}
    Let $\Psi_1$ and $\Psi_2$ be the witnesses in \eqref{eq-diff-QQtilde} and \eqref{eq-diff-RRtilde} respectively. Consider $\widetilde U$ and $U$ (respectively, $\widetilde V$ and $V$) in the singular value decompositions of $\widetilde X$ and $X$ (respectively, $\widetilde Y$ and $Y$). Then $\mathbb P(\widetilde{\mathcal E}_2) = 1-o(1)$, where $\widetilde{\mathcal E}_2$ is defined to be the following event:
    \begin{align}
        \|\widetilde U-U\Psi_1\|_{\operatorname{op}} \leq n^{-0.49} \mbox{ and } \|\widetilde V-V\Psi_2\|_{\operatorname{op}}\leq n^{-0.49} \,. \label{eq-diff-UUtilde-VVtilde} 
    \end{align}
\end{lemma}
\begin{proof}
    By the triangle inequality, we have that 
    \begin{align}
        &\|\widetilde U-U\Psi_1\|_{\operatorname{op}} =\|\widetilde X\widetilde Q\widetilde\Lambda^{-1}-XQ\Psi_1\Lambda^{-1}\|_{\operatorname{op}} \nonumber \\
        \leq \ &\frac{1}{\sqrt{n}}\|\widetilde X\widetilde Q-XQ\Psi_1\|_{\operatorname{op}}+\|\widetilde X\widetilde Q(\widetilde\Lambda^{-1}-\frac{1}{\sqrt{n}}\operatorname{I}_d)\|_{\op}+\|XQ\Psi_1(\Lambda^{-1}-\frac{1}{\sqrt{n}} \operatorname{I}_d) \|_{\op} \,.  \label{eq-diff-U-tilde-U-3-parts}
    \end{align}
    Using \eqref{eq-singular-tildeX} and \eqref{eq-eigen-Lambda}, with probability $1-o(1)$ we have
    \begin{equation*}
       \|\Lambda^{-1}-\frac{1}{\sqrt{n}} \operatorname{I}_d \|_{\operatorname{op}} \leq \frac{3d\widehat{\omega}_n}{n} \mbox{ and }
       \|\widetilde\Lambda^{-1}-\frac{1}{\sqrt{n}} \operatorname{I}_d \|_{\operatorname{op}} \leq \frac{3d\widehat{\omega}_n}{n} \,.
    \end{equation*}
    Therefore, with probability $1-o(1)$ we have
    \begin{equation}{\label{eq-bound-diff-U-tilde-U-part-1}}
        \| \widetilde{X} \widetilde{Q} (\widetilde{\Lambda}^{-1}-\frac{1}{\sqrt{n}} \operatorname{I}_d)\|_{\op} \leq \| \widetilde X \|_{\op} \| \widetilde{\Lambda}^{-1}-\frac{1}{\sqrt{n}} \operatorname{I}_d \|_{\op} \leq 2\sqrt{n} \cdot \frac{3d\widehat{\omega}_n}{n} = \frac{6d\widehat{\omega}_n}{\sqrt{n}} \,.
    \end{equation}
    And similarly
    \begin{equation}{\label{eq-bound-diff-U-tilde-U-part-2}}
       \| X Q \Psi_1(\Lambda^{-1}-\frac{1}{\sqrt{n}} \operatorname{I}_d)\|_{\op} \leq \frac{6d\widehat{\omega}_n}{\sqrt{n}} \,.
    \end{equation}
    Finally we have
    \begin{align}
        \frac{1}{\sqrt{n}}\|\widetilde X\widetilde Q-XQ\Psi_1\|_{\operatorname{op}} &\leq \frac{1}{\sqrt{n}} \Big(\|(X-\widetilde X)\widetilde Q\|_{\op}+\|X(\widetilde Q-Q\Psi_1)\|_{\op} \Big) \nonumber \\
        &\leq \frac{1}{\sqrt{n}} \Big(\|(X-\widetilde X) \|_{\op}+ \| X \|_{\op} \|(\widetilde Q-Q\Psi_1)\|_{\op} \Big) \leq n^{-0.49} \,, \label{eq-bound-diff-U-tilde-U-part-3}
    \end{align}
    where the first inequality follows from triangle inequality, and the third inequality follows from \eqref{eq-diff-X-Xtilde} and \eqref{eq-diff-QQtilde}. Plugging \eqref{eq-bound-diff-U-tilde-U-part-1}--\eqref{eq-bound-diff-U-tilde-U-part-3} into \eqref{eq-diff-U-tilde-U-3-parts} yields that with probability $1-o(1)$ we have $\| \widetilde{U} -U\Psi_1 \|_{\op} \leq n^{-0.49}$.
    The proof of the claim regarding $\| \widetilde{V}-V\Psi_2 \|$ is the same, and thus we omit the details.
\end{proof}
Similar to the case of the dot-product model, we define the following ``good" events which occur with probability $1-o(1)$.
\begin{lemma}\label{lem-good-events-E2-distance-model}
    Define $\widetilde{\mathcal{E}_\diamond}=\mathcal{E}_\diamond\cap\widetilde{\mathcal{E}}_1\cap\widetilde{\mathcal{E}_2}\cap{\widetilde{\mathcal{E}}_3}$, where
    \begin{equation}
        \widetilde{\mathcal{E}}_3:= \Big\{ \, \forall\, \Pi\in\mathfrak S_{n,\epsilon}, \Psi\in\mathcal{S}_d \,\mbox{ we have }\, \|\Pi \widetilde U-\widetilde U\Psi\|_{\operatorname{F}}\geq 8d^3 \widehat{\sigma} \Big\}\,.
    \end{equation}
       Then we have that $\widetilde{\mathcal{E}}_{\diamond}$ holds with probability $1-o(1)$.
\end{lemma}
Provided with Lemma~\ref{lem-good-events-E2-distance-model}, we can finish the proof of Theorem~\ref{thm-almost-exact-recovery-distance} in the same way just as how we derived Theorem~\ref{thm-almost-exact-recovery} from Corollary~\ref{cor-total-good-event-exact-recovery}. The only difference is that we will replace all $U,V$ with $\widetilde{U}, \widetilde{V}$ and thus we omit further details here. We now finish this subsection by proving Lemma~\ref{lem-good-events-E2-distance-model}.
\begin{proof}[Proof of Lemma~\ref{lem-good-events-E2-distance-model}]
    It suffices to prove that $\widetilde{\mathcal{E}}_3$ holds with probability $1-o(1)$. Let $\Psi_1$ to be the witness of \eqref{eq-diff-QQtilde}, and also recall the definition of $\mathcal{E}_3$ in \eqref{eq-def-good-event-2}. Then using the triangle inequality we have 
    \begin{align}
        \| \Pi \widetilde{U} - \widetilde{U} \Psi \|_{\Fop} &= \| \Pi \widetilde{U} \Psi_1 - \widetilde{U} \Psi \Psi_1 \|_{\Fop} 
        = \| \Pi \widetilde{U} \Psi_1 -  \widetilde{U} \Psi_1 \Psi \|_{\Fop} \nonumber \\
        &\geq \| \Pi U - U \Psi \|_{\Fop} - \| \Pi (\widetilde{U} \Psi_1 - U) \|_{\Fop} -  \| (U-\widetilde{U} \Psi_1) \Psi \|_{\Fop} \nonumber \\
        &= \| \Pi U - U \Psi \|_{\Fop} - \| \widetilde{U} \Psi_1 - U \|_{\Fop} - \| U-\widetilde{U} \Psi_1 \|_{\Fop} \,, \label{eq-E3-to-E3tilde}
    \end{align}
    where the first and the third equalities follow from the fact that Frobenius norm is invariant under any orthogonal transformation and the second equality follows from the fact that $\Psi$ and $\Psi_1$ are diagonal matrices.     Recall the definition of $\mathcal{E}_3$ in \eqref{eq-def-good-event-2}. By \eqref{eq-E3-to-E3tilde}, we have that $\widetilde{\mathcal{E}}_2\cap\mathcal{E}_3\subset \widetilde{\mathcal{E}}_3$. Then Lemma~\ref{lem-good-events-E2-distance-model} follows from Lemma~\ref{lem-typical-E3} and Lemma~\ref{lem-diff-UUtilde}.
\end{proof}

\subsection{Proofs of Exact Recovery}
Now we prove Theorems~\ref{thm-exact-recovery} and \ref{thm-exact-recovery-distance} regarding exact recovery. Again we may assume $\Pi^*=\operatorname{I}_n$. Recall that in either of the two cases by applying Theorem~\ref{thm-almost-exact-recovery} or \ref{thm-almost-exact-recovery-distance} we get that with probability $1-o(1)$ we have 
$\widehat{\Pi}_{\operatorname{Umeyama}}\in \mathfrak S_n\setminus\mathfrak S_{n,\epsilon_0}$, where 
\begin{equation}{\label{eq-def-epsilon-0}}
    \epsilon_0 = \frac{ 50 }{ d \log (\omega_n\wedge\log n) } \geq \frac{50}{\log \log n} \,.
\end{equation}
We define
\begin{equation}\label{eq-check-Sn-eps}
\Check{\mathfrak{S}}_{n,\epsilon}:=\mathfrak S_n \setminus (\mathfrak S_{n, \epsilon} \cup \{\operatorname{I}_n\})\,.
\end{equation}
Recall the definition of $\Psi_0$ in $\mathcal{E}_0$. Thus, for Theorem~\ref{thm-exact-recovery}, it suffices to show that
\begin{align*}
    \mathbb{P} \Big( \exists\, \Pi \in \check{\mathfrak{S}}_{n,\epsilon_0}, \Psi \in \mathcal{S}_d\mbox{ such that } \big\langle \Pi U\Psi,V \big\rangle > \big\langle U\Psi_0,V \big\rangle \Big) = o(1).
\end{align*}
We begin with the dot-product model. Define 
\begin{equation}{\label{eq-def-good-event-3}}
    \mathcal E_4 = \Big\{ \arg \max_{\Psi \in \mathcal{S}_d} \big\langle \Pi U \Psi, V \big\rangle = \Psi_0, \ \forall\, \Pi\in\mathfrak S_n\setminus\mathfrak S_{n,\epsilon_0} \Big\} \,.
\end{equation}
In addition, for any $\Pi \in \mathfrak{S}_n$ define $\mathsf{N}(\Pi)=\{ i \in [n] : \pi(i) \neq i \}$, where $\pi$ is the permutation corresponding to $\Pi$. For any set $\mathsf A \subset [n]$, define $X^{\mathsf A}=[ X^{\mathsf A}_1,\ldots,X^{\mathsf A}_d ]$ to be the matrix of $X,Y$ restricted on $\mathsf A$, where $X^{\mathsf A}_k$ is indexed by $\mathsf A$ with entries given by $X^{\mathsf A}_k(i)=X_k(i)$ for each $i \in \mathsf A$. Define $Y^{\mathsf{A}}, Z^{\mathsf A}$ in the same manner. Define
\begin{equation}
\begin{aligned}
    \mathcal E_5 = \Big\{ & 100 \widehat{\sigma} d^{3} \big\| X^{\mathsf N(\Pi)} \big\|_{\operatorname{F}}, 100 \widehat{\sigma} d^{3} \big\| Z^{\mathsf N(\Pi)} \big\|_{\operatorname{F}} \leq \big\| (\operatorname{I}_n-\Pi)X \big\|_{\operatorname{F}}, 
    \, \forall\, \Pi \in \check{\mathfrak{S}}_{n,\epsilon_0} \Big\} \,. \label{eq-def-good-event-4}
\end{aligned}
\end{equation}

\begin{lemma}\label{lem-typical-E4}
    We have that $\Pb(\mathcal{E}_4)=1-o(1)$.
\end{lemma}
\begin{proof}
    Recall the definition of $\mathcal{E}_{\diamond}$ in Lemma~ \ref{lem-total-good-event-almost-exact-recovery}. It suffices to show that given $\mathcal E_{\diamond}$, we have that $\mathcal E_4$ holds with probability $1-o(1)$. 
Denote $\Psi=\operatorname{diag}(s_1,\ldots,s_d)$. Then we have
\begin{align*}
    \big\langle \Pi U \Psi, V \big\rangle = \big\langle [s_1 \Pi u_1, \ldots, s_d \Pi u_d], [v_1,\ldots,v_d] \big\rangle = \sum_{i=1}^{d} s_i \big\langle \Pi u_i,v_i \big\rangle \,.
\end{align*}
So the maximizer $\Psi$ of $\big\langle \Pi U \Psi, V \big\rangle$ must satisfy $s_i = \operatorname{sgn}(\langle \Pi u_i,v_i \rangle)$. It remains to show that with probability $1-o(1)$ we have 
\begin{align}
    \operatorname{sgn}(\langle \Pi u_i,v_i \rangle) = \operatorname{sgn}(\langle u_i,v_i \rangle), \ \forall\, 1 \leq i \leq d, \Pi\in\mathfrak S_n\setminus\mathfrak S_{n,\epsilon_0} \,. \label{eq-goal-uniqueness-maximizer}
\end{align}
Write $\Psi_0$ in \eqref{eq-def-good-event-0-III} by $\Psi_0 = \operatorname{diag}(\psi_1,\ldots, \psi_d) \in \mathcal{S}_d$. For each $1 \leq i \leq d$, we have
$$ 9d^3 \widehat\sigma^2  \geq \big\| U \Psi_0 - V \big\|_{\operatorname{F}}^2 = \sum_{j=1}^{d} \| \psi_j u_j - v_j \|^2 \geq \| \psi_i u_i - v_i \|^2 = 2 \big( 1-\psi_i \langle u_i,v_i \rangle \big)  \,. $$
Hence, $|\langle u_i,v_i\rangle|=1+o(1)$. Therefore, to prove \eqref{eq-goal-uniqueness-maximizer} it suffices to show that with probability $1-o(1)$ we have $|\langle \Pi u_i,v_i \rangle - \langle u_i,v_i \rangle| \leq 0.9$ for all $1 \leq i \leq d$ and $\Pi\in\mathfrak S_n\setminus\mathfrak S_{n,\epsilon_0}$. 
By Cauchy-Schwartz inequality we have that
\begin{align*}
    \big( \langle \Pi u_i,v_i \rangle - \langle u_i,v_i \rangle \big)^2 = \big\langle (\Pi-\operatorname{I}_n) u_i, v_i \big\rangle^2 \leq \big\| (\Pi-\operatorname{I}_n) u_i \big\|^2 \leq 4 \sum_{\pi(\ell) \neq \ell} u_i(\ell)^2\,,
\end{align*}
and thus it suffices to show that with probability $1-o(1)$ we have
\begin{equation}{\label{eq-goal-E3}}
    \sum_{\ell \in \mathsf A} u_i(\ell)^2 \leq \frac{1}{6} , \ \forall\, \mathsf A \subset [n], \#\mathsf A \leq \epsilon_0 n, 1 \leq i \leq d \,.
\end{equation}
For each fixed $1\leq i\leq d$, the vector $u_i$ is uniformly distributed on $S^{n-1}$, which can be represented as $u_i=\frac{1}{\Vert z\Vert_2}z$ where $z$ is an $n$-dimensional vector with i.i.d. standard normal entries. Thus for each fixed $\mathsf A$ with $|\mathsf A|\leq \epsilon_0 n$ we have
\begin{align}
    &\Pb \Bigg( \sum_{\ell \in \mathsf A} u_i(\ell)^2 \geq \frac{1}{6} \Bigg) = \Pb \Bigg( \frac{1}{\|z\|_2^2} \sum_{\ell \in \mathsf A} z_\ell^2 \geq \frac{1}{6} \Bigg) \nonumber  \\
    \leq \ &\Pb \Bigg( \sum_{\ell \in \mathsf A} z_\ell^2 \geq \frac{1}{6} \Vert z\Vert_2^2 , \Vert z\Vert_2^2 \geq n/2 \Bigg) +\exp(-0.1n) \nonumber \\
    \leq \ & \Pb \Bigg( \sum_{\ell\in \mathsf A} z_\ell^2 \geq \frac{n}{12} \Bigg) + \exp(-0.1n) \leq 2\exp(-0.1n)\,.
\end{align}
The enumeration of $(i,\mathsf A)$ with $\# \mathsf A \leq \epsilon_0 n$ and $1 \leq i \leq d$ is bounded by $d \cdot \binom{n}{\epsilon_0 n} =\exp(O(\epsilon_0 n))$. So \eqref{eq-goal-E3} follows from a union bound.
\end{proof}

\begin{lemma}\label{lem-typical-E5}
    We have that $\Pb(\mathcal{E}_5)=1-o(1)$.
\end{lemma}
\begin{proof}

In the following we prove 
\begin{equation}{\label{eq-E5-resp-X}}
    \Pb\Big( \exists\,\Pi \in \check{\mathfrak{S}}_{n,\epsilon_0}\mbox{ such that } 100d^3 \widehat{\sigma} \big\|X^{\mathsf N(\Pi)} \big\|_{\operatorname{F}} \leq \big\|(\operatorname{I}_n-\Pi) X \big\|_{\operatorname{F}} \Big) = o(1) \,.
\end{equation}
and the proof of the corresponding claim for $Z^{\mathsf N(\Pi)}$ is similar. Write $\alpha_n = 100 \widehat{\sigma} d^3$ for simplicity. We shall prove \eqref{eq-E5-resp-X} by taking a union bound on $\Pi \in \check{\mathfrak{S}}_{n,\epsilon_0}$. For $1 \leq \ell \leq \epsilon_0 n$, define
\begin{equation}{\label{eq-def-ENUM}}
    \operatorname{ENUM}(\ell) = \Big\{ \Pi \in \mathfrak{S}_n : d_{\operatorname{H}}(\Pi,\mathrm{I}_n)=\ell \Big\} \,.
\end{equation}
Clearly we have
\begin{equation}{\label{eq-bound-card-ENUM}}
    \# \operatorname{ENUM}(\ell) \leq \binom{n}{\ell} \cdot \ell! \leq n^{\ell} \,.
\end{equation}
In addition, define
\begin{equation}{\label{eq-def-PROB}}
    \operatorname{PROB}(\ell) = \max_{ \Pi \in \operatorname{ENUM}(\ell) } \Big\{ \Pb \big( \| ( \Pi-\mathrm{I}_n ) X \|_{\operatorname{F}}^2 \leq \alpha_n^2 \| X^{\mathsf N(\Pi)} \|_{\operatorname{F}}^2 \big) \Big\} \,.
\end{equation}
Fix any $\Pi \in \operatorname{ENUM}(\ell)$. Let $\pi$ be the permutation corresponding to $\Pi$. Define $\mathfrak C_k(\Pi)$ as the set of $k$-cycles of $\pi$, and let $n_k(\Pi) = \# \mathfrak{C}_k(\Pi)$. For notation convenience, we abbreviate them as $\mathfrak C_k$ and $n_k$. In addition, for any $\mathcal C \in \mathfrak{C}_k$, we regard $\mathcal C$ as a subset of $[n]$ and write $X^{\mathcal C}$ for convenience. Then the event that $\| ( \Pi-\mathrm{I}_n ) X \|_{\operatorname{F}}^2 \leq \alpha_n^2 \| X^{\mathsf N(\Pi)} \|_{\operatorname{F}}^2$ can be written as 
\begin{equation}\label{eq-summand-union-E5}
    \sum_{k\geq 2} \sum_{\mathcal C \in \mathfrak C_k} \sum_{i \in \mathcal C} \sum_{j=1}^d (X_{i,j}-X_{\pi(i),j})^2 \leq \alpha_n^2 \sum_{k\geq 2} \sum_{\mathcal C\in\mathfrak C_k}\sum_{i\in\mathcal C} \sum_{j=1}^dX_{i,j}^2\,.
\end{equation}
Since the random variables within $$\left\{\sum_{i\in\mathcal{C}}(X_{i,j}-X_{\pi(i),j})^2-\alpha_n^2X_{i,j}^2:\mathcal{C}\in\bigcup_{1\leq k\leq n}\mathfrak C_k,\,1\leq j\leq d\right\}$$
are independent, we proceed to derive a tractable form of 
\begin{equation}\label{eq-pre-tractable-qform}
    \sum_{i \in \mathcal C}\big(X_{i,j}-X_{\pi(i),j}\big)^2- \alpha_n^2 \sum_{i\in \mathcal C} X_{i,j}^2	
\end{equation}
for each $k$-cycle $\mathcal C$ and index $j$.
Note that the second item of \eqref{eq-pre-tractable-qform} can be written as $\alpha_n^2 X^{\mathcal C}_j \big( X^{\mathcal C}_j \big)^{\top}$. In addition, the first item of \eqref{eq-pre-tractable-qform} equals to $X^{\mathcal C}_j \mathsf{C}_{\mathsf{circ}}^{\mathcal{C}} \big( X^{\mathcal C}_j \big)^{\top}$, where
\begin{equation}\label{eq-pre-tractatble-qform-matrix-E5}
    \mathsf C_{\mathsf{circ}}^{\mathcal{C}} :=
    \begin{bmatrix}
    2 & -1 & 0 &\cdots & 0 & -1 \\
    -1 & 2 & -1 &\cdots & 0 & 0\\
    0 & -1 & 2 &\cdots & 0 & 0\\
    \vdots & \vdots &\vdots &\ddots &\vdots &\vdots\\
    0 & 0 & 0 &\cdots & 2 & -1\\
    -1 & 0 &0 &\cdots &-1 & 2\\
    \end{bmatrix} \,.
\end{equation}
Note that we may diagonalize $\mathsf C_{\mathsf{circ}}^{\mathcal{C}}$ with $\mathsf C_{\mathsf{circ}}^{\mathcal{C}}=\mathsf{O}_{\mathsf{circ}}^{\mathcal{C}}\mathsf{\Lambda}_{\mathsf{circ}}^{\mathcal{C}}(\mathsf{O}_{\mathsf{circ}}^{\mathcal{C}})^\top$ where $\mathsf{O}_{\mathsf{circ}}^{\mathcal{C}}$ is an orthogonal matrix and $\Lambda_{\mathsf{circ}}^{\mathcal{C}}$ is a diagonal matrix with entries given by $\mathsf{C}_{\mathsf{circ}}^{\mathcal{C}}$'s eigenvalues 
(since $\mathsf C_{\mathsf{circ}}^{\mathcal{C}}$ is a circular matrix)
\begin{equation}\label{eq-eigenvalue-qmatrix-E5}
    \lambda_r^{\mathcal C}=2-2\cos\Big(\frac{2\pi r}{k}\Big) , \quad	0\leq r\leq k-1\,.
\end{equation}
Define the block diagonal matrix $\mathsf O_{\mathsf{circ}}=\operatorname{diag}(\{\mathsf O_{\mathsf{circ}}^{\mathcal{C}}:\,\mathcal{C}\in\cup_{1\leq k\leq n}\mathfrak C_k\})$ and $\mathsf \Lambda_{\mathsf{circ}}=\operatorname{diag}(\{\mathsf \Lambda_{\mathsf{circ}}^{\mathcal{C}}:\,\mathcal{C}\in\cup_{1\leq k\leq n}\mathfrak C_k\})$. For each $1\leq j\leq d$, define $X_j(\Pi)=(X_j^{\mathcal{C}})_{\mathcal{C}\in\cup_{1\leq k\leq n}\mathfrak C_k}$ (which is a permuted version of $X_j$ according to the cycle decomposition of $\Pi$). We define 
$$ 
W_j=X_j(\Pi)\mathsf O_{\mathsf{circ}} \,.
$$
Since $\mathsf O_{\mathsf{circ}}$ is an orthogonal matrix, we have that $(W_{i,j})_{1\leq i\leq n,\,1\leq j\leq d}$ are i.i.d. standard normal random variables. Then we rewrite \eqref{eq-pre-tractable-qform} as

\begin{equation*}
    \sum_{i \in \mathcal C} \lambda_i^{\mathcal C} W_{i,j}^2 - \alpha_n^2 \sum_{i \in \mathcal C} W_{i,j}^2 \,,
\end{equation*} 
and thus \eqref{eq-summand-union-E5} can be written as
\begin{equation}
    \sum_{j=1}^d \sum_{k\geq 2} \sum_{\mathcal C \in \mathfrak C_k} \sum_{i \in \mathcal C} \lambda_i^{\mathcal C} W_{i,j}^2 \leq \alpha_n^2 \sum_{j=1}^d \sum_{k\geq 2} \sum_{\mathcal C\in\mathfrak C_k} \sum_{i\in\mathcal C} W_{i,j}^2\,.
\end{equation}
Note that we have $\lambda_r^\mathcal C\geq 0$ for all cycle $\mathcal C$. In addition, for each $\mathcal C$ we define 
\begin{equation}
    \mathcal{C}^{+} := \big\{ r \in \mathcal C : \lambda_r^\mathcal C \geq 1 + \alpha_n^2 \big\} \,.
\end{equation}
Then it is straightforward from \eqref{eq-eigenvalue-qmatrix-E5} that $\# \mathcal{C}^{+} \geq \lceil \frac{k}{2}\rceil$ for any $\mathcal C \in \mathfrak{C}_k$.
Therefore the probability of \eqref{eq-summand-union-E5} is bounded by
\begin{align}
    &\Pb\Bigg( \sum_{k\geq 2} \sum_{\mathcal C\in\mathfrak C_k}\sum_{i\in\mathcal C^{+}} \sum_{j=1}^d (1+\alpha_n^2) W_{i,j}^2 \leq \alpha_n^2 \sum_{k\geq 2} \sum_{\mathcal C \in \mathfrak C_k} \sum_{i\in\mathcal \mathcal C} \sum_{j=1}^d W_{i,j}^2 \Bigg) \nonumber \\
    = \ &\Pb\Bigg( \frac{3}{8\alpha_n^2} \sum_{k\geq 2} \sum_{\mathcal C\in\mathfrak C_k}\sum_{i\in\mathcal C^{+}} \sum_{j=1}^d W_{i,j}^2 \leq \frac{3}{8}  \sum_{k\geq 2} \sum_{\mathcal C \in \mathfrak C_k} \sum_{i\in \mathcal C\setminus\mathcal{C}^+} \sum_{j=1}^d W_{i,j}^2 \Bigg) \nonumber \\
    \leq \ &\prod_{k\geq 2}\prod_{\mathcal C\in\mathfrak C_k} \prod_{i\in\mathcal C^{+}}\prod_{j=1}^d \mathbb E \Big[ e^{ -3W_{i,j}^2/(8\alpha_n^2)} \Big] \times \prod_{k\geq 2} \prod_{\mathcal C\in\mathfrak C_k} \prod_{i\in\mathcal C\setminus\mathcal C^{+}} \prod_{j=1}^d \mathbb{E} \Big[ e^{ 3W_{i,j}^2/8 } \Big] \nonumber \\
    \leq \ &\Big(\sqrt{ 4\alpha_n^2/(4\alpha_n^2+3) } \Big)^{\sum_{k\geq 2} d n_k \lceil\frac{k}{2}\rceil} \times 2^{\sum_{k\geq 2} d n_k \lfloor\frac{k}{2}\rfloor } \leq \big(\frac{4}{\sqrt{3}}\alpha_n \big)^{\frac{d\ell}{2}} \,.
\end{align}
Here in the first inequality we used the fact that for two independent random variables $\xi$ and $\zeta$, we have that $$\Pb(\xi\leq\zeta)=\Pb(e^{\zeta-\xi}\geq 1)\leq\mathbb E(e^{\zeta-\xi})=\mathbb E(e^{\zeta})\mathbb E(e^{-\xi})\,,$$ and in the second inequality we used $\sum_{k \geq 2} kn_k = \#\{ i:\pi(i) \neq i \}=\ell$. In conclusion, we have shown
\begin{equation}{\label{eq-bound-PROB}}
    \operatorname{PROB}(\ell) \leq \big(\frac{4}{\sqrt{3}}\alpha_n \big)^{\frac{d\ell}{2}} \,.
\end{equation}
Combining \eqref{eq-bound-card-ENUM} and \eqref{eq-bound-PROB}, we get that the left hand side of \eqref{eq-E5-resp-X} is bounded by
\begin{align*}
    \sum_{\ell=1}^{\epsilon_0 n} \#\operatorname{ENUM}(\ell) \cdot \operatorname{PROB}(\ell) &\leq \sum_{\ell=1}^{\epsilon_0 n} n^{\ell} \cdot \big(\frac{4}{\sqrt{3}}\alpha_n \big)^{\frac{d\ell}{2}} \\
    &= \sum_{\ell=1}^{\epsilon_0 n} n^{\ell} \cdot \big( \frac{4}{\sqrt{3}} \cdot 100 \widehat{\sigma} d^3  \big)^{\frac{d\ell}{2}} \overset{\eqref{eq-def-widehat-omega}}{=} o(1) \,.
\end{align*}
This finishes the proof of the lemma.
\end{proof}

\begin{cor}
    {\label{cor-total-good-event-exact-recovery}}
    Define $\mathcal E_{\diamond \diamond}= \mathcal E_{\diamond} \cap \mathcal E_4 \cap \mathcal E_5$.
    We have $\mathbb P(\mathcal E_{\diamond \diamond})=1-o(1)$.
\end{cor}
\begin{proof}
    The corollary follows by combining Lemma~\ref{lem-total-good-event-almost-exact-recovery}, Lemma~\ref{lem-typical-E4} and Lemma~\ref{lem-typical-E5}.
\end{proof}




Now we are ready to prove Theorem~\ref{thm-exact-recovery}.
\begin{proof}[Proof of Theorem~\ref{thm-exact-recovery}.]
Recall the definition of $\Psi_0$ in $\mathcal{E}_0$. It suffices to show that $\mathcal E_{\diamond \diamond}$ implies that
\begin{equation}{\label{eq-exact-recovery-goal-1}}
    \big\langle \Pi U \Psi, V \big\rangle \leq \big\langle U \Psi_0, V \big\rangle, \ \forall\,\Pi \in \check{\mathfrak{S}}_{n,\epsilon_0}, \Psi \in \mathcal{S}_d \,.
\end{equation}
Applying $\mathcal E_4$, it suffices to show that $\mathcal E_{\diamond \diamond}$ implies 
\begin{equation}{\label{eq-exact-recovery-goal-2}}
    \big\langle \Pi U \Psi_0, V \big\rangle \leq \big\langle U \Psi_0, V \big\rangle, \ \forall\,\Pi \in \check{\mathfrak{S}}_{n,\epsilon_0} \,.
\end{equation}
Assume on the contrary that $\langle \Pi U \Psi_0, V \rangle > \langle U \Psi_0, V \rangle$ for some $\Pi \in \check{\mathfrak{S}}_{n,\epsilon_0}$. Then we have (recall \eqref{eq-def-SVD})
\begin{equation}{\label{eq-exact-recovery-relaxation-1}}
    \big\langle (\Pi-\operatorname{I}_n)X Q\Lambda^{-1}\Psi_0, Y R\Sigma^{-1} \big\rangle > 0 \Longrightarrow \big\langle (\Pi-\operatorname{I}_n) XQ \Lambda^{-1} \Psi_0 \Sigma^{-1} R^{\top} , Y \big\rangle \geq 0 \,.
\end{equation}
(Note that we switched from $>$ to $\geq$ above since this does not matter for our proof.) Thus we have
\begin{equation}{\label{eq-exact-recovery-relaxation-2}}
    \big\langle (\operatorname{I}_n-\Pi) X , Y  \big\rangle \leq \big\langle (\operatorname{I}_n-\Pi) X (\operatorname I_d-n Q\Lambda^{-1} \Psi_0 \Sigma^{-1} R^\top), Y \big\rangle \,.
\end{equation}
This shows (recall now $Y=X+\sigma Z$)
\begin{equation}{\label{eq-exact-recovery-relaxation-3}}
    \big\langle (\operatorname{I}_n-\Pi) X , X  \big\rangle \leq \sigma \big| \big\langle (\Pi-\operatorname{I}_n) X , Z \big\rangle \big| + \big\langle (\operatorname{I}_n-\Pi) X (\operatorname I_d-nQ \Lambda^{-1} \Psi_0 \Sigma^{-1} R^\top), Y \big\rangle \,.
\end{equation}
Since the $i$-th row of $(\mathrm{  I}_n-\Pi)X$ is zero if $i \in \mathsf{N}(\Pi)$, we have
\begin{align}
    \big| \big\langle (\operatorname{I}_n-\Pi) X , Z  \big\rangle \big| = \big| \big\langle ((\operatorname{I}_n-\Pi) X)^{\mathsf N(\Pi)} , Z^{\mathsf N(\Pi)} \big\rangle \big| \leq \| (\operatorname{I}_n-\Pi) X \|_{\operatorname{F}} \| Z^{\mathsf N(\Pi)} \|_{\operatorname{F}} \,. \label{eq-exact-recovery-relaxation-4}
\end{align}
Under $\mathcal E_1$ we have $\| \mathrm{I}_d-nQ \Lambda^{-1}\Psi_0 \Sigma^{-1}R^\top \|_{\operatorname{op}} \leq 2 \widehat{\sigma} d^3$, and thus (combining the fact that the $i$'th row of $ (\operatorname{I}_n-\Pi) X (\mathrm{I}_d-nQ \Lambda^{-1}\Psi_0 \Sigma^{-1}R^\top)$ is zero for $i \in \mathsf{N}(\Pi)$)
\begin{align}
    &\big\langle (\operatorname{I}_n-\Pi) X (\mathrm{I}_d-nQ \Lambda^{-1}\Psi_0 \Sigma^{-1}R^\top), Y \big\rangle \nonumber \\
    = \ & \big\langle (\operatorname{I}_n-\Pi) X (\mathrm{I}_d-nQ \Lambda^{-1}\Psi_0 \Sigma^{-1}R^\top), Y^{\mathsf N(\Pi)} \big\rangle \nonumber \\
    \leq \ & \big\| (\operatorname{I}_n-\Pi) X (\mathrm{I}_d-nQ \Lambda^{-1}\Psi_0 \Sigma^{-1}R^\top) \big\|_{\operatorname{F}} \big\| Y^{\mathsf N(\Pi)} \big\|_{\operatorname{F}} \nonumber \\
    \leq \ & \big\| (\operatorname{I}_n-\Pi) X \big\|_{\operatorname{F}} \big\| \mathrm{I}_d-nQ \Lambda^{-1}\Psi_0 \Sigma^{-1}R^\top \big\|_{\operatorname{op}} \big\| Y^{\mathsf N(\Pi)} \big\|_{\operatorname{F}} \nonumber \\
    \leq \ & 2 d^3 \widehat{\sigma} \big\| (\operatorname{I}_n-\Pi) X \big\|_{\operatorname{F}} \big\| Y^{\mathsf N(\Pi)} \big\|_{\operatorname{F}} \,. \label{eq-exact-recovery-relaxation-5}
\end{align}
Plugging \eqref{eq-exact-recovery-relaxation-4} and \eqref{eq-exact-recovery-relaxation-5} into \eqref{eq-exact-recovery-relaxation-3}, we obtain 
\begin{align*}
    \big\langle (\operatorname{I}_n-\Pi) X, X \big\rangle &= \frac{1}{2} \big\| (\operatorname{I}_n-\Pi) X \big\|_{\operatorname{F}}^2 \leq 2 d^3 \widehat{\sigma} \big\| (\operatorname{I}_n-\Pi) X \big\|_{\operatorname{F}} \big( \big\| Y^{\mathsf N(\Pi)} \big\|_{\operatorname{F}} + \big\| Z^{\mathsf N(\Pi)} \big\|_{\operatorname{F}} \big) \\
    &\leq 2 d^3 \widehat{\sigma} \big\| (\operatorname{I}_n-\Pi) X \big\|_{\operatorname{F}} \big( \big\| X^{\mathsf N(\Pi)} \big\|_{\operatorname{F}} + (1+\sigma) \big\| Z^{\mathsf N(\Pi)} \big\|_{\operatorname{F}} \big) \,,
\end{align*}
which contradicts to the event $\mathcal E_5$. This implies \eqref{eq-exact-recovery-goal-1} holds under $\mathcal E_{\diamond \diamond}$.    
\end{proof}

Now we turn to the distance model. It remains to show that
\begin{align}\label{eq-exact-finalgoal-dist}
    \mathbb{P} \Big( \exists\, \Pi\in \check{\mathfrak{S}}_{n,\epsilon_0}, \Psi \in \mathcal{S}_{d}, \big\langle \Pi \widetilde U\Psi,\widetilde V \big\rangle > \big\langle \widetilde U \widetilde{\Psi}_0,\widetilde V \big\rangle \Big) = o(1)\,,
\end{align}
where $\widetilde{\Psi}_0$ is defined in Proposition~\ref{prop-diff-othogonal-matrix-distance}. Define 
\begin{equation}{\label{eq-def-good-event-3-distance}}
    \widetilde{\mathcal E}_4 = \Big\{ \arg \max_{\Psi\in \mathcal S_d} \big\langle \Pi \widetilde U \Psi,\widetilde V \big\rangle = \widetilde\Psi_0, \ \forall\,  \Pi\in\mathfrak S_n\setminus\mathfrak S_{n,\epsilon_0} \Big\} \,.
\end{equation}
In addition, define 
\begin{equation}{\label{eq-def-good-event-4-distance}}
    \begin{aligned}
        \widetilde{\mathcal E}_5 = \Big\{ & 50d^3 \widehat{\sigma} \big\| \widetilde X^{\mathsf N(\Pi)} \big\|_{\operatorname{F}}, 50d^3 \widehat{\sigma} \big\| \widetilde Z^{\mathsf N(\Pi)} \big\|_{\operatorname{F}} \leq \big\| (\operatorname{I}_n-\Pi)\widetilde X \big\|_{\operatorname{F}},  \, \forall\,\Pi \in \check{\mathfrak{S}}_{n,\epsilon_0} \Big\} \,.
    \end{aligned}
\end{equation}

\begin{lemma}{\label{cor-total-good-event-exact-recovery-distance}}
    Define $\widetilde{\mathcal E}_{\diamond \diamond}= \widetilde{\mathcal E}_{\diamond} \cap \widetilde{\mathcal E}_4 \cap \widetilde{\mathcal E}_5$.
    We have $\mathbb P(\widetilde{\mathcal E}_{\diamond \diamond})=1-o(1)$.
\end{lemma}
Again, once Lemma~\ref{cor-total-good-event-exact-recovery-distance} has been established, we can complete the proof of Theorem~\ref{thm-exact-recovery-distance} in the same manner as we derived Theorem~\ref{thm-exact-recovery} from Corollary~\ref{cor-total-good-event-exact-recovery}. The only difference is that we will replace all $X,Y,U,V$ with $\widetilde{X}, \widetilde{Y}, \widetilde{U}, \widetilde{V}$ so we omit further details here. Now we provide the proof of Lemma~\ref{cor-total-good-event-exact-recovery-distance}.
\begin{proof}[Proof of Lemma~\ref{cor-total-good-event-exact-recovery-distance}]
    Firstly we deal with $\widetilde{\mathcal{E}}_4$, the proof is similar to that of $\mathcal{E}_4$. Denote $\Psi=\operatorname{diag}(s_1,\ldots,s_d)$, then
\begin{align*}
    \big\langle \Pi \widetilde U \Psi, \widetilde V \big\rangle = \big\langle [s_1 \Pi \widetilde u_1, \ldots, s_d \Pi\widetilde u_d], [\widetilde v_1,\ldots,\widetilde v_d] \big\rangle = \sum_{i=1}^{d} s_i \big\langle \Pi \widetilde u_i,\widetilde v_i \big\rangle \,.
\end{align*}
So the maximizer $\Psi$ of $\big\langle \Pi \widetilde U \Psi,\widetilde V \big\rangle$ must satisfy $s_i = \operatorname{sgn}(\langle \Pi\widetilde u_i,\widetilde v_i \rangle)$. It remains to show with probability $1-o(1)$ we have 
\begin{align}
    \operatorname{sgn}(\langle \Pi \widetilde u_i,\widetilde v_i \rangle) = \operatorname{sgn}(\langle \widetilde u_i,\widetilde v_i \rangle), \ \forall\, 1 \leq i \leq d, \Pi\in\mathfrak S_n\setminus\mathfrak S_{n,\epsilon_0} \,. \label{eq-goal-uniqueness-maximizer-distance}
\end{align}
Recall the definition of $\widetilde\Psi_0$ in Proposition~\ref{prop-diff-othogonal-matrix-distance} and we write $\widetilde{\Psi}_0 = \operatorname{diag}(\widetilde\psi_1,\ldots,\widetilde\psi_d) \in \mathcal{S}_{d}$. By Corollary~\ref{cor-diff-UV-distance-model}, we have that for each $1 \leq i \leq d$
$$ (4d^3 \widehat{\sigma})^2 \geq \big\| \widetilde U \widetilde\Psi_0 - \widetilde V \big\|_{\operatorname{F}}^2 = \sum_{j=1}^{d} \| \widetilde\psi_j \widetilde u_j - \widetilde v_j \|^2 \geq \| \widetilde\psi_i \widetilde u_i - \widetilde v_i \|^2 = 2 \big( 1- \widetilde\psi_i \langle\widetilde u_i,\widetilde v_i \rangle \big) \,. $$
Thus $|\langle \widetilde u_i,\widetilde v_i\rangle|=1+o(1)$, so to prove \eqref{eq-goal-uniqueness-maximizer-distance} it suffices to show that with probability $1-o(1)$ we have $|\langle \Pi \widetilde u_i,\widetilde v_i \rangle - \langle \widetilde u_i,\widetilde v_i \rangle| \leq 0.9$ for all $1 \leq i \leq d$ and $\Pi\in\mathfrak S_n\setminus\mathfrak S_{n,\epsilon_0}$. 
Note that
\begin{align*}
    &\big( \langle \Pi \widetilde u_i,\widetilde v_i \rangle - \langle \widetilde u_i,\widetilde v_i \rangle \big)^2 = \big\langle (\Pi-\operatorname{I}_n) \widetilde u_i, \widetilde v_i \big\rangle^2\\
    \leq& \big\| (\Pi-\operatorname{I}_n) \widetilde
    u_i \big\|^2 \leq 4 \sum_{\pi(\ell) \neq \ell} \widetilde u_i(\ell)^2\leq 4\sum_{\pi(\ell)\neq\ell}u_i(l)^2+n^{-0.48}\,,
\end{align*}
where the last inequality follows from Lemma~\ref{lem-diff-UUtilde}. By \eqref{eq-goal-E3}, the desired result for $\widetilde{\mathcal{E}}_3$ follows. 

For $\widetilde{\mathcal{E}}_5$, we first apply Chebyshev's inequality on $\|\mathbf F X\|_{\Fop}$ and get that with probability $1-o(1)$
\begin{equation}\label{eq-InminusPi-tildeX}
    \|\widetilde X^{\mathsf N(\Pi)}\|_{\operatorname{F}} \leq \|X^{\mathsf N(\Pi)} \|_{\operatorname{F}}+ \|\mathbf FX\|_{\operatorname{F}} \leq \|X^{\mathsf N(\Pi)} \|_{\operatorname{F}} +d \widehat{\omega}_n^{0.01} \,.
\end{equation}
Since $(\operatorname{I}_n-\Pi)\mathbf{F}=0$, we have 
\begin{equation}
    \label{eq-InminusPi-F}(\operatorname{I}_n-\Pi)\widetilde X=(\operatorname{I}_n-\Pi)X\,.
\end{equation}
Therefore, we have
\begin{align}
    &\Pb \Big( \exists\,\Pi\neq\operatorname{I}_n, 50d^3\widehat{\sigma} \|\widetilde X^{\mathsf N(\Pi)} \|_{\operatorname{F}} \geq \|(\operatorname{I}_n-\Pi) \widetilde X \|_{\operatorname{F}} \Big) \nonumber \\
    \overset{\eqref{eq-InminusPi-tildeX}}{\leq} \ &\Pb \Big( \exists\,\Pi\neq\operatorname{I}_n, 50d^3\widehat\sigma ( \| X^{\mathsf N(\Pi)} \|_{\operatorname{F}} +d\widehat{\omega}_n^{0.01} ) \geq \| (\operatorname{I}_n-\Pi) \widetilde{X} \|_{\operatorname{F}} \Big) \nonumber \\
    \overset{\eqref{eq-InminusPi-F}}{=} \ &\Pb \Big( \exists\,\Pi\neq\operatorname{I}_n, 50d^3\widehat\sigma ( \| X^{\mathsf N(\Pi)} \|_{\operatorname{F}} +d\widehat{\omega}_n^{0.01} ) \geq \| (\operatorname{I}_n-\Pi)X \|_{\operatorname{F}} \Big) \nonumber \\
    \leq \ &\Pb \Big( \exists\,\Pi\neq\operatorname{I}_n, 100d^3\widehat\sigma \|X^{\mathsf N(\Pi)} \|_{\operatorname{F}} \geq \|(\operatorname{I}_n-\Pi)X \|_{\operatorname{F}} \Big) \nonumber \\
    &+ \Pb\Big( \exists\,\Pi\neq\operatorname{I}_n, 100d^4 \widehat{\sigma} \widehat{\omega}_n^{0.01} \geq \|(\operatorname{I}_n-\Pi)X \|_{\operatorname{F}} \Big) \nonumber \,.
\end{align}
The first item in the last inequality is $o(1)$ by \eqref{eq-def-good-event-4} and Corollary~\ref{cor-total-good-event-exact-recovery}. The second item is also $o(1)$ whose proof is analogous to that for $\mathcal{E}_4$ in Corollary~\ref{cor-total-good-event-exact-recovery} and is thus omitted. As for $\widetilde Z^{\mathsf N(\Pi)}$, we reach the same conclusion as \eqref{eq-InminusPi-tildeX}. 
This completes the proof for $\widetilde{\mathcal{E}}_5$. 
\end{proof}

\section*{Acknowledgement} We thank Jian Ding for stimulating discussions as well as for providing us with many helpful suggestions for this manuscript, and Shuangping Li for stimulating discussions on an early stage of the project. Shuyang Gong and Zhangsong Li are partially supported by the National Key R$\&$D program of China (Project No. 2023YFA1010103) and the NSFC Key Program (Project No. 12231002).

\bibliographystyle{plain}

\end{document}